\documentclass[reqno,english, ]{amsart}
\usepackage[T1]{fontenc}
\usepackage[latin9]{inputenc}
\usepackage{babel}
\usepackage{float}
\usepackage{units}
\usepackage{amstext}
\usepackage{amsthm}
\usepackage{amssymb}
\usepackage{wasysym}
\usepackage[unicode=true,pdfusetitle,
 bookmarks=true,bookmarksnumbered=false,bookmarksopen=false,
 breaklinks=false,pdfborder={0 0 1},backref=false,colorlinks=false]
 {hyperref}

\makeatletter
\theoremstyle{plain}
\newtheorem*{conjecture*}{\protect\conjecturename}
\theoremstyle{plain}
\newtheorem{thm}{\protect\theoremname}
\theoremstyle{plain}
\newtheorem{lem}[thm]{\protect\lemmaname}
\theoremstyle{definition}
\newtheorem{defn}[thm]{\protect\definitionname}
\theoremstyle{plain}
\newtheorem{cor}[thm]{\protect\corollaryname}

\usepackage{ytableau}
\usepackage{calrsfs}
\usepackage{graphicx}
\usepackage{color}
\usepackage{perpage}
\usepackage{upquote}
\usepackage{bbm}
\usepackage[shortlabels]{enumitem}
\usepackage{romanbar}
\usepackage{breakurl}

\MakePerPage{footnote}
\gdef\SetFigFontNFSS#1#2#3#4#5{} 
\gdef\SetFigFont#1#2#3#4#5{} 
\def\clap#1{\hbox to 0pt{\hss#1\hss}}

\DeclareMathOperator{\Tr}{Tr}

\DeclareMathOperator{\ch}{ch}
\DeclareMathOperator{\hght}{ht}
\DeclareMathOperator{\Sym}{Sym}

\definecolor{myblue}{rgb}{0.09,0.32,0.44} 
\hypersetup{pdfborder={0 0 0},pdfborderstyle={},colorlinks=true,linkcolor=myblue,citecolor=myblue,urlcolor=blue}

\theoremstyle{remark}
\newtheorem*{qst*}{Question}
\newtheorem*{rmrks*}{Remarks}

\newlength{\tempindent} 
\newcommand{\lazyenum}{
\setlength{\tempindent}{\parindent} 
\begin{enumerate}[leftmargin=0cm,itemindent=0.7cm,labelwidth=\itemindent,labelsep=0cm,align=left,label=\arabic*)]
\setlength{\parskip}{\smallskipamount}
\setlength{\parindent}{\tempindent}
}

\newcommand{\brmul}{\discretionary{\mbox{$\,\cdot$}}{}{}}

\ifpdf
\def\afs#1#2{\href{#1}{\nolinkurl{#2}}}
\else
\def\afs#1#2{\burlalt{#1}{#2}}
\fi

\@ifundefined{showcaptionsetup}{}{%
 \PassOptionsToPackage{caption=false}{subfig}}
\usepackage{subfig}
\makeatother

\providecommand{\conjecturename}{Conjecture}
\providecommand{\corollaryname}{Corollary}
\providecommand{\definitionname}{Definition}
\providecommand{\lemmaname}{Lemma}
\providecommand{\theoremname}{Theorem}

\begin{document}
\title[The quantum Heisenberg ferromagnet]{The mean-field quantum Heisenberg ferromagnet via representation
theory}
\begin{abstract}
We use representation theory to write a formula for the magnetisation
of the quantum Heisenberg ferromagnet. The core new result is a spectral
decomposition of the function $\alpha_{k}2^{\alpha_{1}+\dotsb+\alpha_{n}}$
where $\alpha_{k}$ is the number of cycles of length $k$ of a permutation.
In the mean-field case, we simplify the formula further, arriving
at a closed-form expression for the magnetisation, which allows to
analyse the phase transition.
\end{abstract}

\author{Gil Alon and Gady Kozma}
\address{GA: Department of Mathematics and Computer Science, The Open University
of Israel, 4353701 Raanana, Israel}
\address{GK: Department of Mathematics and Computer Science, The Weizmann Institute
of Science, 76100 Rehovot, Israel.}
\maketitle

\section{Introduction}

The quantum Heisenberg ferromagnet is one of the simplest multiparticle
quantum system one may imagine. Particles do not move, they are fixed
in lattice positions. Each particle is endowed with a spin and the
interactions are spin interactions. Since spins are described by $2\times2$
matrices the entire system may be described by a $2^{n}\times2^{n}$
matrix, removing the need to involve operator theory, as is common
in quantum mechanics.

In a landmark paper, Dyson, Lieb and Simon analysed the quantum Heisenberg
\emph{antiferromagnet} using an approach known as quantum reflection
positivity \cite{DLS78}. Surprisingly, the ferromagnetic version
could not be analysed using the same approach and is still open (the
classical version of reflection positivity is less sensitive to model
details, see \cite{FSS76,FILS,FILS2}).

An initially unrelated line of research, inspired by classical physics,
was that of particle systems. In 1970, Spitzer \cite{S70} defined
the \emph{exclusion} process: given a graph $G$ (Spitzer was mainly
interested in the case that $G$ is $\mathbb{Z}^{d}$) put particles
on its vertices, coloured white and black, and Poisson clocks on the
edges. When a clock rings, the two particles on the sides of that
edge are exchanged. A similar process, when the particles do not have
two types but rather each particle is unique is called the \emph{interchange}
process. In 1981 Diaconis and Shahshahani \cite{DS}, motivated by
problems in card shuffling, proved that the \emph{mixing time }of
the interchange process on the \emph{complete graph} on $n$ vertices
is $\frac{1}{2}\log n$. In physics' parlance, studying the complete
graph is called \emph{mean-field}, so one may say that Diaconis and
Shahshahani studied the mean-field interchange process. 

The connection between the quantum Heisenberg ferromagnet and the
interchange process was first derived using physics arguments in \cite{P76},
and then made rigorous by works of Conlon and Solovej \cite{CS91}
and Tóth \cite{Toth}. The inverse temperature of the quantum system
translates to the time the interchange process is run, and various
physical quantities translate to various questions about the cycle
structure of the permutation of the interchange particles (one such
quantity will be described in details below). Interestingly, the most
physically relevant setting, that of fixed temperature and system
size going to infinity, translates to studying the interchange process
at constant time, in particular before its mixing time.

This connection generated a lot of interest since the interchange
process is subject to many different attack vectors, both probabilistic
and algebraic. A notable result achieved by a purely probabilistic
approach is that of Schramm \cite{S05}, who did a very fine analysis
of the cycle structure of the interchange process in the mean-field
case. His approach was recently applied to the hypercube \cite{KMU16}
and to the Hamming graph \cite{AKM}. The case that $G$ is a tree
was studied by Angel \cite{A03} and Hammond \cite{H13,H15}. In particular,
\cite{H15} shows a \emph{sharp} phase transition (when the tree has
sufficiently high degrees), an impressive result given that the system
has no inherent monotonicity.

The algebraic approach starts from the observation that the interchange
process is a random walk on $\Sym_{n}$, the group of permutations
of $n$ elements. Hence representation theory may be applied to it,
somewhat similarly to the classic application of Fourier transform
to studying simple random walk. This is most effective in the mean-field
case, since then representation theory gives full linearisation of
the problem. This fact is behind the work of Diaconis and Shahshahani.
It also allows to study the interchange process at constant time,
see \cite{BK}. Nevertheless, representation theory can also be used
for general graphs. For example, it gives a simple formula for the
probability that the interchange process is one large cycle, see \cite{AK}.

In the discussion so far we have ignored an important issue of \emph{weighting}
that appears when translating from the quantum Heisenberg ferromagnet
to the interchange process. This weighting means that all the results
for the interchange process that we have quoted do \emph{not} translate
directly to the quantum model, and may serve only as a heuristic guide.
To explain this any further we need to delve into the details of the
interchange process, and we will do so now. (The quantum system will
be defined in details later, in \S\ref{subsec:quantum})

\subsection{The interchange process}

Let us now define the interchange process (a.k.a.\ the stirring process)
formally, directly defining it as a random walk on $\Sym_{n}$. We
start with the walker at the identity permutation and run it in continuous
time. Every edge of the interaction graph $G$ is attached to a Poisson
clock of rate 1, and when the clock corresponding to an edge $(i,j)$
rings, the position of the walker at that time $\pi(t)$ is changed
to $(ij)\pi(t)$ i.e.\ a transposition is composed from the left.
Notice the attractive feature of using continuous time: for every
vertex $i$, $\pi(t)(i)$ is a continuous-time random walk on $G$.
Readers familiar with the discrete-time formulation of the interchange
process should note that time $t$ corresponds to approximately $tn$
steps of the discrete time process.

Next, let us define two physical quantities, the \emph{free energy}
$Z$ and the \emph{expected squared magnetisation} $m$, already translated
(via \cite{Toth}) to the interchange process. For a permutation $\pi$,
denote by $\alpha_{k}(\pi)$ the number of cycles of length $k$ in
$\pi$ and by $\alpha(\pi)=\sum_{i}\alpha_{i}(\pi)$ the total number
of cycles in $\pi$. Define
\begin{equation}
Z(t)=\mathbb{E}(2^{\alpha(\pi(t))})\qquad m^{2}(t)=\frac{1}{Z(t)}\mathbb{E}\Big(\Big(\sum_{k}k^{2}\alpha_{k}(\pi(t)\Big)2^{\alpha(\pi(t))}\Big).\label{eq:Zmalpha}
\end{equation}
In \S \ref{subsec:quantum} we will explain the connection of $Z$
and $m$ to the quantum system, but for now (\ref{eq:Zmalpha}) will
suffice as the definition.

The most important conjecture about the quantum Heisenberg ferromagnet
is that in dimension $d\ge3$ there is a \emph{phase transition} in
the behaviour of the magnetisation. Let us state formally a version
of it:
\begin{conjecture*}
For every $d\ge3$ there is a critical value, denoted by $t_{c}$,
with the following property. Let the interaction graph $G$ be $\{1,\dotsc,l\}^{d}$
where $\{1,\dotsc,l\}$ is given the graph structure of the cycle
and the power is the usual (``square'') graph product. Then
\begin{align*}
t & <t_{c}\implies m(t)\apprle\sqrt{n}\\
t & >t_{c}\implies m(t)\apprge n
\end{align*}
\end{conjecture*}
(since $n$ denotes for us the size of the graph, we have $n=l^{d}$
in the conjecture). Here and below, the notation $X\apprle Y$ means
that there is some constant $C>0$, which may depend on $d$ and on
$t$, but not on $n$, such that $X\le CY$ (the conjecture is as
$n$ (or $l$) goes to $\infty$). It is easy to show that for $t$
sufficiently small the behaviour is indeed as stated i.e.\ $m(t)\apprle\sqrt{n}$.
What is not known is that there exists an ordered phase \emph{at all}
i.e.\ that for some $t$ sufficiently large we have $m(t)\apprge n$.

The factors $2^{\alpha}$ that appear in (\ref{eq:Zmalpha}) are the
weights that we mentioned in the previous section. The result of Tóth
from \cite{Toth} is that the naturally defined magnetisation translates
precisely to the quantity $m$ from (\ref{eq:Zmalpha}). But from
the point of view of the interchange process, it is natural to simply
ask if $\sum k^{2}\alpha_{k}$ undergoes a phase transition in the
time $t$. And indeed, all results we mentioned in the previous section
are of this form. For example, Berestycki and Durrett \cite{BD06}
show that, for $G$ being the complete graph, there is a $t_{c}$
such that
\[
t<\frac{t_{c}}{n}\implies\mathbb{E}\big(\sum k^{2}\alpha_{k}(\pi(t))\big)\apprle n\qquad t>\frac{t_{c}}{n}\implies\frac{1}{n}\mathbb{E}\big(\sum k^{2}\alpha_{k}(\pi(t))\big)\to\infty
\]
(we compare $t$ to $t_{c}/n$ rather than to $t_{c}$ because of
the high degree of the complete graph, which speeds the process $n$-fold).
Schramm \cite{S05} improved this by showing that at $t>t_{c}/n$,
$\mathbb{E}(\sum k^{2}\alpha_{k}(\pi(t)))\apprge n^{2}$ (both results
give more information than just the expectation, for example Schramm
gave a description of the distribution of the sizes of the large cycles).
Despite all this progress, the case where the interaction graph G
is $\{1,\dotsc,l\}^{d}$ is still open, even for the question without
the weights.

Our purpose in this paper is twofold:
\begin{enumerate}
\item Prepare for an algebro-analytic attack on the non mean-field case.
\item Give an alternative and potentially more precise analysis of the mean-field
case.
\end{enumerate}
We start with the mean-field results, as they do not require representation
theory to state. Our thoughts about how the non mean-field case may
be attacked via these results are best discussed after some preliminaries
and we will do so in \S \ref{subsec:algebro-analytic}. As above,
the high degree requires a scaling of the time:
\begin{thm}
\label{thm:phase transition}When the interaction graph is the complete
graph we have
\begin{align*}
t & <\frac{2}{n}\implies m(t)\apprle\sqrt{n}\\
t & >\frac{2}{n}\implies m(t)\apprge n.
\end{align*}
\end{thm}

This theorem is not new, it can be inferred from a different paper
of Tóth, \cite{Toth90}, which attacks the mean-field quantum problem
using an approach not related to the interchange process (\cite{Toth90}
discusses the free energy, but the magnetisation should follow from
it by differentiation). Penrose gave a second proof at about the same
time \cite{P91}. Another proof of theorem \ref{thm:phase transition}
was given recently in \cite{B15,B16}, valid when the term $2^{\alpha}$
in (\ref{eq:Zmalpha}) is replaced by $\theta^{\alpha}$ for an arbitrary
$\theta>0$, using the techniques of Schramm \cite{S05}. However,
our techniques give a closed formula for the magnetisation:
\begin{thm}
\label{thm:precise}Under the same conditions as theorem \ref{thm:phase transition},
\[
m_{n}^{2}(t)=n\frac{\sum_{b,k}k\psi(b,k,t)}{\sum_{b,k}\psi(b,k,t)}
\]
the sums running over $k$ from $1$ to $n$ and $b$ from $0$ to
$\lfloor(n-k)/2\rfloor$, and where $\psi$ is defined, for $b>0$
by 
\begin{align*}
\psi(b,k,t) & =(a+1-b)\binom{n}{k-1,a+1,b}\cdot e^{-t(ab+b)}\Biggl[\frac{(a+1-b)y^{a+b+2}(1-y)^{k-1}}{k}\\
 & \qquad\qquad+\;y^{b}(a+1-b+k)\int_{y}^{1}x^{a+1}(1-x)^{k-1}dx\\
 & \qquad\qquad+\;y^{a+1}(a+1-b-k)\int_{y}^{1}x^{b}(1-x)^{k-1}dx\Biggr]
\end{align*}
and for $b=0$ by
\[
\psi(0,k,t)=(a+1)\binom{n}{k}\left[y^{n-k+1}(1-y)^{k-1}+k\int_{y}^{1}x^{n-k}(1-x)^{k-1}dx\right]
\]
and where $a=n-k-b$ and $y=e^{-tk}$.
\end{thm}

Concluding theorem \ref{thm:phase transition} from theorem \ref{thm:precise}
is an exercise and we do it in the end. Certainly, a finer analysis
of theorem \ref{thm:precise} would give a very detailed picture of
the phase transition, but we prefer to stick with the relatively rough
theorem \ref{thm:phase transition} as our focus here is the proof
of theorem \ref{thm:precise}. The proof uses representation theory,
so let us turn to this topic now.

\subsection{\label{subsec:Repesentations}Representations of the symmetric group}

Group representations are a non-commutative analogue of the Fourier
transform, so let us use notation suggestive of it. Let $f:\Sym_{n}\to\mathbb{C}$
be some function and let $R:\Sym_{n}\to\mathrm{GL}(V_{R})$ be some
representation of $\Sym_{n}$. We denote 
\[
\widehat{f}(R)=\sum_{\pi\in S_{n}}f(\pi)R(\pi)
\]
so $\widehat{f}(R)$ is a linear map from $V_{R}$ to $V_{R}$. We
have an analogue of Parseval's formula (see, e.g., \cite[theorem 4.1]{D88}),
\[
\langle f,g\rangle=\sum_{R}\frac{d_{R}}{n!}\Tr(\widehat{f}(R)\widehat{g}(R)^{*})
\]
where the sum is over all (equivalence classes of) \emph{irreducible}
unitary representations of $\Sym_{n}$; $d_{R}=\dim V_{R}$ i.e.\ the
dimension of the space the representation $R$ acts upon; and $\langle f,g\rangle=\sum_{\pi\in\Sym_{n}}f(\pi)g(\pi)$
i.e.\ is not normalised. To apply this to, for example, the partition
function, $Z(t)=\mathbb{E}(2^{\alpha(\pi(t))})$ denote by $p_{t}(\pi)$
the probability that $\pi(t)=\pi$ and then 
\begin{equation}
Z(t)=\langle2^{\alpha},p_{t}\rangle=\sum_{R}\frac{d_{R}}{n!}\Tr(\widehat{2^{\alpha}}(R)\widehat{p_{t}}(R)^{*}).\label{eq:Z by Parseval}
\end{equation}
Now, $2^{\alpha}$ is a \emph{class function} i.e.\ a function which
depends only on the cycle structure of the permutation. Hence Schur's
lemma tells us that $\widehat{2^{\alpha}}(R)$ is a scalar matrix
for any irreducible $R$, though it does not give the value of the
scalar. As for $p_{t}$, it is a class function only in the mean-field
case. This is what makes the mean-field case amenable to analysis,
and in fact, $\widehat{p_{t}}$ was calculated explicitly by Diaconis
and Shahshahani \cite{DS}. 

Thus theorem \ref{thm:precise} will follow once we calculate $\widehat{2^{\alpha}}$
(for $Z$) and $\widehat{\alpha_{k}2^{\alpha}}$ (for $m$). To state
these, recall that the irreducible representations of $S_{n}$ can
be indexed by partitions of $n$ (see e.g.\ \cite{JKbook}). We denote
by $R_{\lambda}$ the irreducible representation of $S_{n}$ corresponding
to the partition $\lambda$. We use the standard notation $a^{b}$
for a repetition in a partition, so, for example, $[3,1^{3}]$ is
the partition $6=3+1+1+1$. The case $\widehat{2^{\alpha}}$ is not
difficult (we give the details below):
\begin{equation}
\widehat{2^{\alpha}}(\lambda)=\frac{n!}{d_{\lambda}}I_{\lambda}\begin{cases}
a-b+1 & \lambda=[a,b]\\
0 & \text{otherwise}
\end{cases}\label{eq:2alpha at}
\end{equation}
where $I_{\lambda}$ is the identity matrix of the vector space $V_{\lambda}$
(recall that the values of the non-commutative Fourier transform are
matrices). Here and below we use the short-hand notation $V_{\lambda}=V_{R_{\lambda}}$,
$d_{\lambda}=d_{R_{\lambda}}$ and $\widehat{f}(\lambda)=\widehat{f}(R_{\lambda})$.

For the function $\alpha_{k}2^{\alpha}$ we need to introduce some
further notation. Recall that a \emph{skew diagram }is the set difference
$\lambda\backslash\mu$ of two Young diagrams $\lambda,\mu$, and
a \emph{border strip} is a connected skew diagram which does not contain
a $2\times2$ square. When we say about a skew diagram that it is
connected we mean connection through edges, so $\raisebox{0.25em}{\ytableausetup{boxsize=0.5em}\ydiagram{2,1}}
\setminus
\raisebox{0.25em}{\ytableausetup{boxsize=0.5em}\ydiagram{1,0}}
=
\raisebox{0.25em}{\ytableausetup{boxsize=0.5em}\ydiagram{1+1,1}}$ is not considered to be connected. For a skew diagram $\mu$, the
height $\hght(\mu)$ is defined to be the number of its rows.
\begin{thm}
\label{thm:nokronecker}We have $\widehat{\alpha_{k}2^{\alpha}}(\lambda)=\frac{n!}{d_{\lambda}}a_{\lambda,k}I_{\lambda}$
where$$
a_{\lambda,k}=
\frac{2}{k}
\begin{cases}
  (a-b+1)\cdot(-1)^{\hght(\lambda\setminus [a,b])+1} & 
    \parbox{4cm}{$\exists[a,b]\dashv n-k$ such that 
    $\lambda\setminus[a,b]$ is a border strip}\\
  0&\text{otherwise.}
\end{cases}
$$
\end{thm}

It is easy to check that $a$ and $b$ are uniquely defined by $\lambda$
and $k$, if they exist. Of course, the notation $[a,b]\dashv n-k$
i.e.\ $[a,b]$ is a partition of $n-k$, is nothing but $a\ge b\ge0$
and $a+b=n-k$. We note as a corollary of theorem \ref{thm:nokronecker}
that the Fourier transform is supported on diagrams of the form $[a,b,c,1^{d}]$.

Let us sketch how theorem \ref{thm:precise} follows from theorem
\ref{thm:nokronecker}. Define the laplacian $\Delta:S_{n}\to\mathbb{R}$
by
\begin{equation}
\Delta=\sum_{i\sim j}\nabla_{ij}\qquad\nabla_{ij}(\pi)=\begin{cases}
1 & \pi=\text{id}\\
-1 & \pi=(ij)\\
0 & \text{otherwise,}
\end{cases}\label{eq:Lap}
\end{equation}
where id is the identity permutation; and where $i\sim j$ means that
$i$ and $j$ are neighbours in the interaction graph $G$. Then $p_{t}=e^{-t\Delta}$
where the exponentiation is in the convolution algebra over $\Sym_{n}$
(or in the group ring $\mathbb{R}[\Sym_{n}]$, which is the same).
Since Fourier transform translates convolution into product we also
get $\widehat{p_{t}}(\lambda)=e^{-t\widehat{\Delta}(\lambda)}$ where
this time this is simply exponentiation of matrices. By (\ref{eq:Z by Parseval})
we get
\[
Z\stackrel{\textrm{{(\ref{eq:Z by Parseval})}}}{=}\sum_{R}\frac{d_{R}}{n!}\Tr(\widehat{2^{\alpha}}(R)\widehat{p_{t}}(R)^{*})\stackrel{\textrm{(\ref{eq:2alpha at})}}{=}\sum_{b=0}^{n/2}(n-2b+1)\Tr(\widehat{p_{t}}([n-b,b])).
\]
Denote by $\rho_{1}(\lambda),\dotsc,\rho_{d_{\lambda}}(\lambda)$
the eigenvalues of $\widehat{\Delta}(\lambda)$ and get
\begin{align*}
Z & =\sum_{b=0}^{n/2}(n-2b+1)\Tr(\exp(-t\widehat{\Delta}([n-b,b])))\\
 & =\sum_{b=0}^{n/2}(n-2b+1)\sum_{i=1}^{d_{[n-b,b]}}\exp(-t\rho_{i}([n-b,b])).
\end{align*}
A similar formula holds for $m$,
\[
m^{2}=\frac{1}{Z}\sum_{k=1}^{n}k^{2}\sum_{\lambda\dashv n}a_{\lambda,k}\sum_{i=1}^{d_{\lambda}}\exp(-t\rho_{i}(\lambda))
\]
where $a_{\lambda,k}$ are given by theorem \ref{thm:nokronecker}.

The calculation so far was for any $G$. We now assume $G$ is the
complete graph. In this case, $\Delta$ is in the center of the group
ring of $\Sym_{n}$, so by Schur's lemma, $\rho_{1}(\lambda)=\rho_{2}(\lambda)=\dotsb=\rho_{d_{\lambda}}(\lambda)$,
and we simplify our formulas writing
\[
Z=\sum_{b=0}^{n/2}d_{[n-b,b]}\rho([n-b,b])\qquad m=\sum_{k}k^{2}\sum_{\lambda}a_{\lambda,k}d_{\lambda}\exp(-t\rho(\lambda))
\]
where $\rho(\lambda)$ denotes the common value. Next, $\rho(\lambda)$
can be calculated by examining the trace of $\Delta$, leading to
\begin{equation}
\rho(\lambda)=\binom{n}{2}\left(1-\frac{\chi_{\lambda}\left((12)\right)}{d_{\lambda}}\right)\label{eq:rho from char ratio}
\end{equation}
where $\chi_{\lambda}$ is the character of $R_{\lambda}$ i.e.\ $\chi_{\lambda}(\sigma)=\Tr R_{\lambda}(\sigma)$.
The value $\frac{\chi_{\lambda}\left((12)\right)}{d_{\lambda}}$ is
called the \emph{character ratio} and has a formula, mainly due to
Frobenius, which is quoted in \cite{DS}. Thus we get explicit formulas
for $Z$ and $m$.

\subsubsection{\label{subsec:algebro-analytic}}

We end this part of the introduction with a few remarks on the case
where the interaction graph is not the complete graph. Here the main
difficulty is to learn something about $\widehat{p_{t}}$, or equivalently
about $\widehat{\Delta}$. Since we eventually are only interested
in the trace of $e^{-t\widehat{\Delta}}$, we see that we need to
understand the eigenvalues of $\widehat{\Delta}(\lambda)$, and not
for any $\lambda,$ but just for $\lambda$ where the coefficients
in theorem \ref{thm:nokronecker} are non-zero, i.e.\ for $\lambda$
of the form $[a,b,c,1^{d}]$. Denote the eigenvalues which correspond
to a given graph $G$ and diagram $\lambda$ by $\rho_{i}(\lambda;G)$,
ordered in increasing order.

Let us mention two interesting bounds on these eigenvalues. The first
is a bound for the minimum eigenvalue,
\begin{equation}
\rho_{1}(\lambda;\{1,\dotsc,l\}^{d})\apprge\frac{1}{nl^{2}}\rho_{1}(\lambda;K_{n}).\label{eq:min eigen}
\end{equation}
(Of course, $\rho(\lambda;K_{n})$ is given by (\ref{eq:rho from char ratio})).
This follows from a comparison argument (an operator version of the
classic multi-commodity flow argument). See \cite{DSC93},\cite[\S 5]{CGS15}
or \cite{AK19}, and see also \cite{SST} for more about the minimal
eigenvalue. (In \cite{AK19} we demonstrate an example where the information
about the minimal eigenvalue is enough to prove the existence of a
phase transition.) Probably, the estimate (\ref{eq:min eigen}) is
tight up to the value of the constant. A second fact, known only for
two-rows Young diagrams, is that the statistics of the $\rho_{i}([n-k,k])$
is the same as that of sums of $k$-tuples of $\rho_{i}([n-1,1])$,
which are simply the eigenvalues of the Laplacian on the interaction
graph $G$ \cite{T10}. Thus there is good control of two scales,
the smallest and the largest, and what is missing to prove the conjecture
is control of the eigenvalues in intermediate scales.

\subsection{\label{subsec:Frobenius}The Frobenius map}

We will now delve a little deeper into the algebraic aspects of this
work, namely the proof of theorem \ref{thm:nokronecker}. Before starting,
let us make a remark on one natural approach to prove theorem \ref{thm:nokronecker}
which we could not make work. Recall that we wish to calculate $\widehat{\alpha_{k}2^{\alpha}}$
or, alternatively, to present $\alpha_{k}2^{\alpha}$ as a combination
of irreducible characters. As already mentioned, $\widehat{2^{\alpha}}$
is supported on Young diagrams of the form $[a,b]$ (\ref{eq:2alpha at}).
Moreover, $\widehat{\alpha_{k}}$, as shown in \cite{AK}, is supported
on Young diagrams of the form $[a,b,1^{c}]$. It may seem natural
to try and deduce $\widehat{\alpha_{k}2^{\alpha}}$ directly from
these decompositions. However, the problem of writing the product
(also known as the Kronecker product) of two irreducible characters
as a linear combination of irreducible characters is highly nontrivial:
It has been studied for more than eighty years. Many partial results
are known, but not one suitable for this case. In \cite{Rosas} Rosas
found the decomposition of a Kronecker product for a two-row diagram
with both a hook-shaped diagram (i.e.\ a diagram of the form $[a,1^{b}]$)
and with a second two-row diagram. This falls short of the case in
question. Our proof thus takes a different route, which we now sketch.

As in \cite{AK}, we express the decomposition of $\widehat{\alpha_{k}2^{\alpha}}$
as a sum of irreducible characters via the Frobenius characteristic
map $\ch$. Let us recall this classical object. The Frobenius map
is a function from the space of class functions on $\Sym_{n}$ to
the ring of symmetric functions, which maps $\chi_{\lambda}$, the
character of the representation $\lambda$, to the Schur polynomial
$S_{\lambda}$. See \cite[Definition 7.10.1 or \S\S 7.10, 7.15]{Stanley2}
for the Schur polynomials. We will prove that
\[
\ch(\alpha_{k}2^{\alpha})=u_{k}\cdot\ch(2^{\alpha})
\]
where $u_{k}=\sum_{i}x_{i}^{k}$. Consequently, the decomposition
of $u_{k}\ch(2^{\alpha})$ to Schur polynomials can be obtained from
that of $\ch(2^{\alpha})$ by a variant of the Murnaghan-Nakayama
rule. Theorem \ref{thm:nokronecker} follows as a consequence.

Another point of the calculation we wish to stress is the use of \emph{hook
numbers}. The hook numbers of a partition $n\vdash\lambda$ are defined
by 
\[
\lambda^{(i)}=\lambda_{i}+\hght(\lambda)-i.
\]
It turns out that all of $a_{\lambda}$, $d_{\lambda}$ and, in the
mean-field case, $\rho(\lambda)$ can be conveniently expressed in
terms of the hook numbers of $\lambda$. This change of variables
yields a considerable simplification of the expression for $\mathbb{E}\left(\alpha_{k}\left(\pi(t)\right)2^{\alpha(\pi(t))}\right)$.
We will evaluate the sums in this expression (in a similar way to
the calculation in \cite{BK}), and obtain a closed form expression
in terms of some incomplete Beta integrals.

\subsection{The quantum model\label{subsec:quantum}}

We now return to the starting point of this paper, the quantum Heisenberg
ferromagnet, and describe it in detail. We are given $n$ interacting
particles. The interaction scheme is defined by a graph $G$ on $n$
vertices, and two particles interact if there is an edge between the
corresponding vertices. Each particle is spin-$\frac{1}{2}$ and interactions
are spin interactions. Hence each particle corresponds to a vector
in $\mathbb{C}^{2}$ and the state space of the entire system is $\otimes_{i=1}^{n}\mathbb{C}^{2}$.
Recall the \emph{Pauli matrices} $\sigma^{x}=\left(\begin{smallmatrix}0 & 1\\
1 & 0
\end{smallmatrix}\right)$, $\sigma^{y}=\left(\begin{smallmatrix}0 & -i\\
i & 0
\end{smallmatrix}\right)$ and $\sigma^{z}=\left(\begin{smallmatrix}1 & 0\\
0 & -1
\end{smallmatrix}\right)$ and define the Pauli matrix at particle $i$ by 
\[
\sigma_{i}^{x}=\underbrace{I\otimes\dotsb\otimes I}_{i-1\text{ times}}\otimes\sigma^{x}\otimes\underbrace{I\otimes\dotsb\otimes I}_{n-i\text{ times}}
\]
and ditto for $\sigma^{y}$ and $\sigma^{z}$. The \emph{Hamiltonian}
is now 
\[
H=-{\textstyle \frac{1}{4}}\sum_{i\sim j}\sigma_{i}^{x}\sigma_{j}^{x}+\sigma_{i}^{y}\sigma_{j}^{y}+\sigma_{i}^{z}\sigma_{j}^{z}
\]
where the notation $i\sim j$ means that $i$ and $j$ are neighbours
in the interaction graph $G$. The partition function at temperature
$T$ is 
\[
Z=Z_{n}(\beta)=\Tr\exp(-\beta H)
\]
where $\beta=1/T$. The expected square magnetisation is
\begin{equation}
m_{n}^{2}(\beta)=\frac{1}{Z_{n}(\beta)}\Tr\left(\exp(-\beta H)\Big(\sum_{i}\sigma_{i}^{z}\Big)^{2}\right).\label{eq:mquantum}
\end{equation}
The result of \cite{Toth} imply that the $m$ defined above is exactly
the one defined in (\ref{eq:Zmalpha}), with $t=\frac{1}{2}\beta$,
i.e.\ the inverse temperature in the quantum model becomes the time
in the interchange model (the quantum and interchange definitions
of $Z$ differ by a constant).

This might be a good point to indicate a small issue regarding the
definition of magnetisation. The magnetisation as defined in \cite{Toth}
is the \emph{residual magnetisation}, i.e., an external field $h$
is applied (mathematically, the term $h\sum\sigma_{i}^{z}$ is added
to the Hamiltonian), the magnetisation is calculated as a function
of $h$, $n$ is taken to infinity and then $h$ is taken to $0$.
This residual magnetisation (denote it by $m^{*}$) can also be described
by the interchange process, 
\begin{equation}
m^{*}(\beta)=\frac{1}{2}\lim_{M\to\infty}\lim_{n\to\infty}\frac{1}{Z_{n}(\beta)}\sum_{k=M+1}^{n}\frac{k}{n}\mathbb{E}(\alpha_{k}2^{\alpha(\pi(\beta))})\label{eq:mstar}
\end{equation}
assuming the limits exist, see \cite[(5.2)]{Toth}. (unfortunately,
while the formula (\ref{eq:mstar}) for $m^{*}$ can be found in \cite{Toth}
explicitly, the equivalent formula for $m$, i.e.\ the equivalence
of our two definitions for $m$, (\ref{eq:Zmalpha}) and in (\ref{eq:mquantum})
is not written as such in \cite{Toth}. Nevertheless the argument
is sufficiently similar to the arguments of \cite{Toth} that we feel
we may omit it).

It is easy to see that $m(t)\apprge n$ as $n\to\infty$ implies that
$m^{*}(t)>0$, since the first says that there are cycles of linear
size, while the latter says that there is some ``mass'' in cycles
of size larger than constant. As this is quite standard, we postpone
the proof to the appendix.

\subsection{Recap of notation}

Throughout this paper, we will analyse the continuous time interchange
process on $n$ particles, with respect to the complete graph.  See
\cite{AK} for more details on the interchange process. We denote
by $\pi(t)$ the permutation at time $t\geq0$.

For a permutation $\pi\in S_{n}$, we denote by $\alpha_{k}(\pi)$
the number of cycles of length $k$ in $\pi$, and by $\alpha(\pi)=\sum_{i}\alpha_{i}(\pi)$
the total number of cycles in $\pi$. We will denote by $c_{i}(\pi)$
the size of the cycle of $\pi$ containing $i$.

For a partition $n\vdash\lambda$, we denote by $R_{\lambda}$ the
irreducible representation of $\Sym_{n}$ associated with $\lambda$,
by $d_{\lambda}$ its dimension, and by $\chi_{\lambda}$ the corresponding
irreducible character. We denote by $a_{\lambda,k}$ the numbers from
theorem \ref{thm:nokronecker}. When $k$ is clear from context, we
will often remove it, writing simply $a_{\lambda}$.

We let $\ch$ be the Frobenius characteristic map. It is a function
from the space of class function on $\Sym_{n}$ to the ring of symmetric
functions, which maps $\chi_{\lambda}$ to the Schur polynomial $S_{\lambda}$.

\section{The character decomposition}

The main step in the proof of theorem \ref{thm:nokronecker} is the
calculation of $\ch(\alpha_{k}2^{\alpha})$. We state it as a result.
\begin{thm}
\label{thm:decomposition}We have
\begin{enumerate}
\item \label{enu:just2alpha}
\[
\ch(2^{\alpha})=\sum_{n\vdash[a,b]}(a-b+1)S_{[a,b]}
\]
where $b$ is allowed to be $0$, in which case $[a,b]$ means $[a]$.
In other words, we sum over all couples $a$, $b$ such that $a\ge b\ge0$
and $a+b=n$.
\item \label{enu:alphak2alpha}For $1\leq k\leq n$,
\[
\ch(\alpha_{k}2^{\alpha})=\frac{2}{k}u_{k}\sum_{n-k\vdash[a,b]}(a-b+1)S_{[a,b]}
\]
where $u_{k}=\sum_{i}x_{i}^{k}$.
\end{enumerate}
\end{thm}

\begin{proof}
[Proof of (\ref{enu:just2alpha})]The polynomial $\sum_{\pi\in S_{n}}x^{\alpha(\pi)}$
is the generating function of the Stirling numbers of the first kind,
and is equal to $\prod_{i=0}^{n-1}(x+i)$ (see \cite[proposition 1.3.4]{Stanley1}).
In particular, we have $\sum_{\pi\in S_{n}}2^{\alpha(\pi)}=(n+1)!$.
For any $n\vdash\lambda=[\lambda_{1},\lambda_{2},\dotsc,\lambda_{r}]$
(with $\lambda_{1}\geq\dotsb\geq\lambda_{r}>0)$, consider the corresponding
partition of $\{1,2,\dotsc,n\}$ to sets $A_{\lambda}^{i}$ of sizes
$\lambda_{1},\dotsc,\lambda_{r}$:
\[
A_{\lambda}^{i}=\mathbb{Z}\cap\Big(\sum_{j=1}^{i-1}\lambda_{j},\sum_{j=1}^{i}\lambda_{j}\Big]
\]
and let $T_{\lambda}\cong\prod\Sym_{\lambda_{i}}$ be the group of
permutations in $\Sym_{n}$ preserving all the sets $A_{\lambda}^{i}$.
Let $M_{\lambda}$ be the sum of all monomials which can be obtained
from $\prod x_{i}^{\lambda_{i}}$ by a permutation of the variables
$\left\{ x_{i}\right\} $. We will use the following formula from
\cite[lemma 2]{AK} (valid for any class function $f$ on $\Sym_{n}$):
\begin{equation}
\ch(f)=\sum_{n\vdash\lambda}\left(\frac{1}{|T_{\lambda}|}\sum_{\pi\in T_{\lambda}}f(\pi)\right)M_{\lambda}\label{eq:ChFormula}
\end{equation}
It follows that 
\[
\ch(2^{\alpha})=\sum_{n\vdash\lambda}M_{\lambda}\prod_{i}(\lambda_{i}+1)
\]
Clause (\ref{enu:just2alpha}) will thus be proved once we show that
\begin{equation}
\sum_{n\vdash\lambda}M_{\lambda}\prod_{i}(\lambda_{i}+1)=\sum_{n\vdash[a,b]}(a-b+1)S_{[a,b]}\label{eq:Pieri}
\end{equation}
To see (\ref{eq:Pieri}) we first claim that 
\[
\sum_{n\vdash\lambda}M_{\lambda}\prod_{i}(\lambda_{i}+1)=\sum_{i=0}^{n}S_{[i]}S_{[n-i]}
\]
Indeed, this is proved by comparing the coefficient of $\prod x_{i}^{\lambda_{i}}$
in both sides: It is clearly $\prod(\lambda_{i}+1)$ on the left hand
side, and it is also $\prod(\lambda_{i}+1)$ on the right hand size,
since $S_{[i]}$ is the sum of all monomials of degree $i$ (see e.g.
\cite{AK} for a proof) and $\prod x_{i}^{\lambda_{i}}$ can be decomposed
in exactly $\prod(\lambda_{i}+1)$ ways as a product of two monomials. 

Finally, by Pieri's rule (see \cite[theorem 7.15.7 on page 339]{Stanley2}),
$S_{[i]}S_{[n-i]}$ is the sum of Schur polynomials of two-line and
one-line diagrams, in which $S_{[a,b]}$ is a summand exactly when
$b\leq i\leq a$. This shows (\ref{eq:Pieri}) and hence clause (\ref{enu:just2alpha})
of the theorem.\phantom\qedhere
\end{proof}

\begin{proof}
[Proof of (\ref{enu:alphak2alpha})]For any $\pi\in T_{\lambda}$,
let $\alpha_{k}^{(i)}(\pi)$ be the number of $k$-cycles in the $\Sym_{\lambda_{i}}$-component
of $\pi$ i.e.\ in the restriction of $\pi$ to $A_{\lambda}^{i}$.
Clearly, 
\[
\sum_{\pi\in T_{\lambda}}\alpha_{k}(\pi)2^{\alpha(\pi)}=\sum_{\pi\in T_{\lambda}}\sum_{i:\lambda_{i}\geq k}\alpha_{k}^{(i)}(\pi)2^{\alpha(\pi)}=\sum_{i:\lambda_{i}\geq k}\sum_{\pi\in T_{\lambda}}\alpha_{k}^{(i)}(\pi)2^{\alpha(\pi)}
\]
To compute the inner sum, let us look at all the $k$-size subsets
$A\subseteq A_{\lambda}^{i}$, and for each such $A$ let $T_{\lambda,A}$
be the set of permutations $\pi\in T_{\lambda}$ that preserve $A$
and $\pi|_{A}$ is a $k$-cycle. Since there are $(k-1)!$ $k$-cycles
in $\Sym_{k}$, we have
\[
\sum_{\pi\in T_{\lambda,A}}2^{\alpha(\pi)}=2(k-1)!\left(\prod_{j\neq i}(\lambda_{j}+1)!\right)(\lambda_{i}-k+1)!
\]
Summing over all the possible $A$'s, we get 
\[
\sum_{\pi\in T_{\lambda}}\alpha_{k}^{(i)}(\pi)2^{\alpha(\pi)}=\binom{\lambda_{i}}{k}2(k-1)!\left(\prod_{j\neq i}(\lambda_{j}+1)!\right)(\lambda_{i}-k+1)!
\]
Summing over $i$:
\[
\sum_{\pi\in T_{\lambda}}\alpha_{k}(\pi)2^{\alpha(\pi)}=2\sum_{i:\lambda_{i}\geq k}\binom{\lambda_{i}}{k}\left(\prod_{j\neq i}(\lambda_{j}+1)!\right)(\lambda_{i}-k+1)!(k-1)!
\]
Hence
\begin{align*}
\frac{1}{|T_{\lambda}|}\sum_{\pi\in T_{\lambda}}\alpha_{k}(\pi)2^{\alpha(\pi)} & =2\sum_{i:\lambda_{i}\geq k}\binom{\lambda_{i}}{k}\left(\prod_{j\neq i}(\lambda_{j}+1)\right)\frac{(\lambda_{i}-k+1)!(k-1)!}{\lambda_{i}!}\\
 & =\frac{2}{k}\sum_{i:\lambda_{i}\geq k}\prod_{j}\begin{cases}
(\lambda_{j}+1) & j\neq i\\
(\lambda_{j}+1-k) & j=i
\end{cases}
\end{align*}
We conclude by (\ref{eq:ChFormula}), that
\begin{align*}
\ch(\alpha_{k}2^{\alpha}) & =\frac{2}{k}\sum_{n\vdash\lambda}M_{\lambda}\sum_{i:\lambda_{i}\geq k}\prod_{j}\begin{cases}
(\lambda_{j}+1) & j\neq i\\
(\lambda_{j}+1-k) & j=i
\end{cases}\\
 & =\frac{2}{k}u_{k}\sum_{n-k\vdash\lambda}M_{\lambda}\prod_{j}(\lambda_{j}+1)
\end{align*}
The last equality follows by comparing the coefficient of $\prod x_{i}^{\lambda_{i}}$
on both sides. The result now follows from (\ref{eq:Pieri}).
\end{proof}
The formula in theorem \ref{thm:decomposition} expresses $\ch(\alpha_{k}2^{\alpha})$
as a product of $u_{k}$ and a linear combination of Schur polynomials.
We will use a Murnaghan-Nakayama type formula to further simplify
the expression and present it as a linear combination of Schur polynomials.
Recall from the introduction the notions of a skew diagram and a border
strip, and the notation $\hght(\mu)$. Let $E_{k}(\mu)$ be the set
of Young diagrams $\lambda$ such that $\lambda\supseteq\mu$ and
$\lambda\backslash\mu$ is a border strip of size $k$. We quote the
following formula from Stanley's \emph{Enumerative Combinatorics}
(see \cite[Theorem 7.17.1]{Stanley2}, but note that Stanley's $\hght$
differs from ours by 1):

\begin{equation}
u_{k}S_{\mu}=\sum_{\lambda\in E_{k}(\mu)}(-1)^{\hght(\lambda\backslash\mu)+1}S_{\lambda}\label{eq:StanleyFormula}
\end{equation}
We therefore have
\[
\ch(\alpha_{k}2^{\alpha})=\frac{2}{k}\sum_{n-k\vdash\mu=[a,b]}(a-b+1)\sum_{\lambda\in E_{k}(\mu)}(-1)^{\hght(\lambda\backslash\mu)+1}S_{\lambda}
\]
concluding the proof of theorem \ref{thm:nokronecker}.\qed

\section{The expectation of \texorpdfstring{$\alpha_{k}2^{\alpha}$}{alpha k times 2 to the alpha}
at time \texorpdfstring{$t$}{t}}

Let us see how the decomposition in theorem \ref{thm:nokronecker}
helps us calculate $\mathbb{E}\big(\alpha_{k}\left(\pi(t)\right)\brmul2^{\alpha(\pi(t))}\big)$.
We have the following general lemma:
\begin{lem}
\label{lem:exp_char}Let $\chi_{\lambda}$ be an irreducible character
of $S_{n}$, $R_{\lambda}$ the corresponding representation, $B=\left(b_{ij}\right)$
a symmetric $n\times n$ matrix with positive entries outside the
diagonal, $\left(\pi^{B}(t)\in S_{n}\right)_{t\geq0}$ the continuous
time interchange process with rates $(b_{ij})$, and $\rho_{1},\dotsc,\rho_{d}$
the eigenvalues of $\sum_{i<j}b_{ij}(\text{id}-R_{\lambda}((ij)))$.
Then 
\[
\mathbb{E}(\chi_{\lambda}(\pi^{B}(t)))=\sum_{i}e^{-t\rho_{i}}
\]
\end{lem}

\begin{proof}
The arguments in the proof of \cite[lemma 5]{AK} hold in this generality.
For the sake of brevity, we will not repeat them.
\end{proof}
Let us get back to the analysis of the mean field case, $\left(\pi(t)\right)_{t\geq0}$,
for which the rates are $a_{ij}=1$. Let $n\vdash\lambda$. As mentioned
in \S \ref{subsec:Frobenius}, the eigenvalues of $\widehat{\Delta}(\lambda)$
are all equal to 
\[
\rho(\lambda)=\binom{n}{2}-\binom{n}{2}\frac{\chi_{\lambda}\left(\left(12\right)\right)}{d_{\lambda}}.
\]

\begin{defn}
For any box $(i,j)$ of the diagram $\lambda$, the \emph{content
}of the box is defined by
\[
c((i,j))=i-j
\]
Here and below, $i$ is the row index and $j$ is the position inside
the row, both starting from $1$.
\end{defn}

\begin{lem}
Let $n\vdash\lambda$. We have
\[
\binom{n}{2}\frac{\chi_{\lambda}\left((12)\right)}{d_{\lambda}}=-\sum_{(i,j)\in\lambda}c((i,j))
\]
\end{lem}

\begin{proof}
By the formula in \cite[lemma 7]{DS},
\begin{align*}
\binom{n}{2}\frac{\chi_{\lambda}\left((12)\right)}{d_{\lambda}} & =\frac{1}{2}\sum_{i=1}^{r}\left((\lambda_{i}-i)(\lambda_{i}-i+1)-i(i-1)\right)=\\
 & =\sum_{i=1}^{r}\frac{\lambda_{i}(\lambda_{i}+1)}{2}-\sum_{i=1}^{r}i\lambda_{i}=\\
 & =\sum_{(i,j)\in\lambda}j-\sum_{(i,j)\in\lambda}i=-\sum_{(i,j)\in\lambda}c((i,j)).\qedhere
\end{align*}
\end{proof}
We therefore have

\begin{equation}
\rho(\lambda)=\binom{n}{2}+\sum_{(i,j)\in\lambda}c((i,j))\label{eq:character_ratio}
\end{equation}
and (by theorem \ref{thm:nokronecker} and lemma \ref{lem:exp_char}),
\begin{equation}
\mathbb{E}\left(\alpha_{k}\left(\pi(t)\right)2^{\alpha(\pi(t))}\right)=\frac{2}{k}\sum_{n-k\vdash\mu=[a,b]}(a-b+1)\sum_{\lambda\in E_{k}(\mu)}(-1)^{ht(\lambda\backslash\mu)+1}d_{\lambda}e^{-t\rho(\lambda)}\label{eq:exp_formula}
\end{equation}

\section{The caterpillar\label{sec:The-caterpillar}}

For a Young diagram $\mu=[\mu_{1},\dotsc,\mu_{r}]$, we have denoted
by $E_{k}(\mu)$ the set of Young diagrams $\lambda$ which contain
$\mu$, and such that $\lambda\backslash\mu$ is a border strip of
size $k$. Here is for example the set $E_{3}([5,3]):$

\begin{figure}[H]
\hfill\subfloat{\ytableausetup{smalltableaux}
\begin{ytableau} 
*(white)  & *(white)  & *(white)  & *(white)  & *(white) \\
*(white)  & *(white) & *(white) \\
*(green) \\
*(green) \\
*(green) 
\end{ytableau} }\hfill{}\subfloat{\ytableausetup{smalltableaux}
\begin{ytableau} 
*(white)  & *(white)  & *(white)  & *(white)  &  *(white)\\
*(white)  & *(white)  &  *(white)\\
*(green) & *(green)\\
*(green) 
\end{ytableau} }\hfill{}\subfloat{\ytableausetup{smalltableaux}
\begin{ytableau} 
*(white)  & *(white)  & *(white)  & *(white) & *(white)   \\
*(white)  & *(white) & *(white)  \\
*(green) & *(green) & *(green)
\end{ytableau} }\hfill{}\subfloat{\ytableausetup{smalltableaux}
\begin{ytableau} 
*(white)  & *(white)  & *(white)  & *(white)  &  *(white) & *(green) & *(green) & *(green) \\
*(white) & *(white) & *(white)
\end{ytableau} }\hfill

\caption{$E_{3}([5,2])$}
\end{figure}
Pictorially, we like to think of $E_{k}(\mu)$ as the result of the
following process: We start with the diagram $[\mu_{1},\mu_{2},\dotsc,\mu_{r},1^{k}]$.
We view this diagram as the diagram $\mu$ with a caterpillar with
$k$ square segments, lying below it. The top square $(r+1,1)$ is
the head of the caterpillar. At each stage the caterpillar moves forward
one step. Each segment moves to the place occupied by the next one,
whereas the head moves either up or right, so that the caterpillar
clings to $\mu$. The caterpillar stops when it lies entirely in the
first row (i.e. when the diagram is $[\mu_{1}+k,\mu_{2},\dotsc,\mu_{r}]$.
Some of the diagrams we obtain are not Young diagrams, but the set
of Young diagrams obtained in this process is precisely $E_{k}(\mu)$.

\begin{figure}[H]
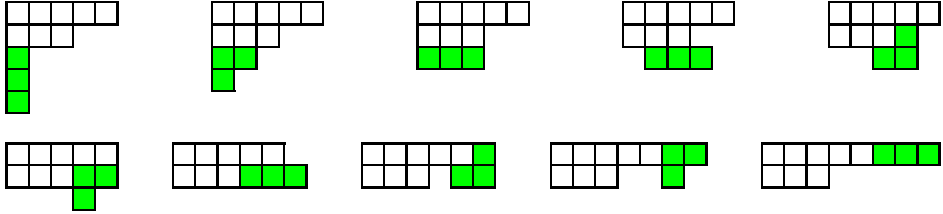

\subfloat{\ytableausetup{smalltableaux}
\begin{ytableau} 
*(white)  & *(white)  & *(white)  & *(white)  & *(white) \\
*(white)  & *(white)  & *(white) \\
*(green) \\
*(green) \\
*(green) 
\end{ytableau} }\hfill{}\subfloat{\ytableausetup{smalltableaux}
\begin{ytableau} 
*(white)  & *(white)  & *(white)  & *(white)  &  *(white)\\
*(white)  & *(white)  &  *(white)\\
*(green) & *(green)\\
*(green) 
\end{ytableau} }\hfill{}\subfloat{\ytableausetup{smalltableaux}
\begin{ytableau} 
*(white)  & *(white)  & *(white)  & *(white) & *(white)   \\
*(white)  & *(white) & *(white)  \\
*(green) & *(green) & *(green)
\end{ytableau} }\hfill{}\subfloat{\ytableausetup{smalltableaux}
\begin{ytableau} 
*(white)  & *(white)  & *(white)  & *(white) & *(white)  \\
*(white)  & *(white) & *(white)  \\
\none & *(green) & *(green) & *(green)
\end{ytableau} }\hfill{}\subfloat{\ytableausetup{smalltableaux}
\begin{ytableau} 
*(white)  & *(white)  & *(white)  & *(white) & *(white)  \\
*(white)  & *(white) & *(white) & *(green) \\
\none & \none & *(green) & *(green)
\end{ytableau} }

\subfloat{\ytableausetup{smalltableaux}
\begin{ytableau} 
*(white)  & *(white)  & *(white)  & *(white) & *(white)  \\
*(white)  & *(white) & *(white) & *(green) & *(green)\\
\none & \none & \none & *(green)
\end{ytableau} }\hfill{}\subfloat{\ytableausetup{smalltableaux}
\begin{ytableau} 
*(white)  & *(white)  & *(white)  & *(white)  & *(white)  \\
*(white)  & *(white)  & *(white) & *(green) & *(green) & *(green) \\
\end{ytableau} }\hfill{}\subfloat{\ytableausetup{smalltableaux}
\begin{ytableau} 
*(white)  & *(white)  & *(white)  & *(white)  & *(white) & *(green) \\
*(white)  & *(white)  & *(white) & \none & *(green) & *(green)  \\
\end{ytableau} }\hfill{}\subfloat{\ytableausetup{smalltableaux}
\begin{ytableau} 
*(white)  & *(white)  & *(white)  & *(white)  & *(white) & *(green)& *(green) \\
*(white)  & *(white)  & *(white) & \none & \none & *(green)   \\
\end{ytableau} }\hfill{}\subfloat{\ytableausetup{smalltableaux}
\begin{ytableau} 
*(white)  & *(white)  & *(white)  & *(white)  &  *(white) & *(green) & *(green) & *(green) \\
*(white) & *(white) & *(white)
\end{ytableau} }

\caption{Caterpillar moves}
\end{figure}
We may narrow down the list of caterpillar positions by requiring
that the tail of the caterpillar is aligned to the left. For any $1\leq i\leq r+k$,
we look at the caterpillar configuration for which the tail lies at
$(i,\mu_{i}+1$) (where $\mu_{j}=0$ for $j>r$). We may obtain this
configuration as follows: We start with the diagram $[\mu_{1},\dotsc,\mu_{i}+k,\dotsc,\mu_{r}]$
(when $i>r$ this means $[\mu_{1},\dotsc,\mu_{r},0^{i-r-1},k]$).
This is not necessarily a Young diagram, as the sequence $\mu_{1},\dotsc,\mu_{i}+k,\dotsc,\mu_{r}$
may not be nonincreasing. We perform a sequence of moves to make it
a Young diagram. In each move, we replace a pair of row lengths $\mu_{j},\mu_{j+1}$
(where $\mu_{j+1}\geq\mu_{j}+2$) with the pair $\mu_{j+1}-1,\mu_{j}+1$
(pictorially, wrapping the $j+1$\textsuperscript{th} row around
the corner above it). We stop when no move is applicable.

Here is an illustration for $\mu=[3,2]$, $k=5$ and $i=4$:

\begin{figure}[H]
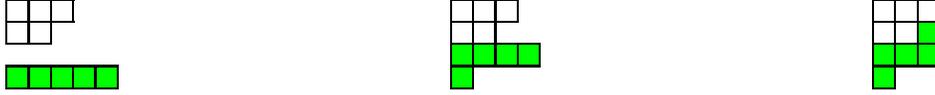

\subfloat{\ytableausetup{smalltableaux}
\begin{ytableau} 
*(white)  & *(white)  & *(white)  \\
*(white)  & *(white) \\
\none \\
*(green) & *(green) & *(green) & *(green) & *(green)  \\
\end{ytableau} }\hfill{}\subfloat{\ytableausetup{smalltableaux}
\begin{ytableau} 
*(white)  & *(white)  & *(white)  \\
*(white)  & *(white) \\
*(green) & *(green) & *(green) & *(green) \\
*(green)  \\
\end{ytableau} }\hfill{}\subfloat{\ytableausetup{smalltableaux}
\begin{ytableau} 
*(white)  & *(white)  & *(white)  \\
*(white)  & *(white) & *(green) \\
*(green) & *(green) & *(green)  \\
*(green)  \\
\end{ytableau}  }

\caption{The wrapping process}
\end{figure}
We now formalize the above description. 
\begin{defn}
We call a diagram of the form $\mu=[\mu_{1},\dotsc,\mu_{r}]$ \emph{a
pre-diagram }if the numbers $\mu_{1},\dotsc,\mu_{r}$ are nonnegative
integers, $\mu_{r}\neq0$, and we have 
\[
|\{i\in\{1,\dotsc,r-1\}:\mu_{i+1}>\mu_{i}\}|\leq1.
\]

Note that any Young diagram is also a pre-diagram.
\end{defn}

Let $R_{j}$ be the move which replaces a pre-diagram $\mu=[\mu_{1},\dotsc,\mu_{r}]$
with $[\mu_{1},\dotsc,\linebreak[0]\mu_{j-1},\mu_{j+1}-1,\mu_{j}+1,\mu_{j+2},\dotsc,\mu_{r}]$.
This move is only applicable when $j\leq r$ and $\mu_{j+1}\geq\mu_{j}+2$.
Note that the resulting sequence is still a pre-diagram.
\begin{lem}
\label{lem:wrap_process}Let $\mu=[\mu_{1},..,\mu_{r}]$ be a pre-diagram.
Consider the following process, called the wrapping process, which
generates a sequence of pre-diagrams $\mu(0),\linebreak[0]\mu(1),\dotsc$:
We start with $\mu(0)=\mu$. Given $\mu(j),$ $\mu(j+1)$ is obtained
from it by one of the moves $R_{l}$. The process terminates when
no move can be applied. Then:

\begin{enumerate}
\item At each stage, there is either one or no move that can be applied.
Thus, the sequence $\mu(0),\mu(1),\dotsc$ is uniquely determined.
\item The sequence $\mu(0),\mu(1),\dotsc$ is finite.
\end{enumerate}
\end{lem}

\begin{proof}
The first assertion is obvious, since for a pre-diagram, there is
at most one index $j$ for which $\mu_{j+1}>\mu_{j}$. The second
assertion is obvious as well, since the index $j$ for $\mu_{j+1}>\mu_{j}$
(if it exists) drops by 1 after each application of a move.
\end{proof}
Let us denote by $Y(\mu)$ the final pre-diagram achieved by the wrapping
process starting from $\mu$. 

Further, for a Young diagram $\mu=[\mu_{1},\dotsc,\mu_{r}]$, and
positive integers $i$ and $k$, let $F_{k,i}(\mu)$ be the diagram

\[
F_{k,i}(\mu)=\begin{cases}
[\mu_{1},\dotsc,\mu_{i}+k,\dotsc,\mu_{r}] & i\leq r\\{}
[\mu_{1},\dotsc,\mu_{r},0^{i-r-1},k] & i>r
\end{cases}
\]
Note that $F_{k,i}(\mu)$ is a pre-diagram.
\begin{lem}
\label{lem:wrap_process_2}Let $n-k\vdash\mu=[\mu_{1},\dotsc,\mu_{r}]$,
$k>0$, $1\leq i\leq r+k$, and let $\delta=Y(F_{k,i}(\mu))$. Then:
\begin{enumerate}
\item $\delta$ is a Young diagram if and only if there exists $\lambda\in E_{\mu,k}$
such that the lowest row of $\lambda\backslash\mu$ is positioned
at $i$. In that case, such $\lambda$ is unique: $\lambda=\delta$.
\item Let $m$ be the total number of moves in the wrapping process on $F_{k,i}(\mu)$.
Then $m=\hght(\delta\backslash\mu)-1$.
\end{enumerate}
\end{lem}

\begin{proof}
Let $\lambda=F_{k,i}(\mu)$ and let $\lambda=\lambda(0),\lambda(1),\dotsc,\lambda(m)=\delta$
be the stages of the wrapping process. It can be shown by induction
that for all $j$, 
\begin{equation}
\lambda(j)=[\mu_{1},\dotsc,\mu_{i-j-1},(\mu_{i}+k-j),\mu_{i-j}+1,\dotsc,\mu_{i-1}+1,\mu_{i+1},\dotsc,\mu_{r}]\label{eq:j_th_step}
\end{equation}
and $u_{i}+k-j\geq u_{i-j}+1$. The only move that can be applied
on $\lambda(j)$ (if any) is $R_{i-j-1}$. The process terminates
when $R_{i-j-1}$ can not be applied to $\lambda(j)$, i.e., when
\[
\mu_{i}+k-j\leq\mu_{i-j-1}+1
\]
which is equivalent to

\[
\mu_{i-j-1}+j\geq\mu_{i}+k-1
\]
hence,
\begin{equation}
\mu_{i-m-1}+m\geq\mu_{i}+k-1\label{eq:mu_ineq}
\end{equation}

Let $B_{\mu,i,k}$ be the set of pre-diagrams $\beta$ such that $\beta\backslash\mu$
is a border strip of size $k$, whose lowest row is at $i$. We claim
that $B_{\mu,i,k}=\{\lambda(0),\dotsc,\lambda(m)\}$. Indeed, any
$\beta\in B_{\mu,i,k}$ must be of the form $[\mu_{1},\dotsc,\mu_{i-j-1},\beta_{i-j},\dotsc,\beta_{i},\mu_{i+1},\dotsc,\mu_{r}]$
for some $j\geq0$ and some numbers $\beta_{i-j}\geq\beta_{i-j+1}\dotsc\geq\beta_{i}$
satisfying $\beta_{l}>\mu_{l}$ for all $l$, and 
\begin{equation}
\sum_{l=i-j}^{i}(\beta_{l}-\mu_{l})=k.\label{eq:total_k}
\end{equation}
By the border strip conditions, we must have for all $i-j<l\leq i$,
$\beta_{l}\geq\mu_{l-1}+1$ (otherwise $\beta\backslash\mu$ would
not be connected), and $\beta_{l}<\mu_{l-1}+2$ (otherwise we would
have $\beta_{l-1}\geq\beta_{l}\geq\mu_{l-1}+2$ and then $\beta\backslash\mu$
would contain a $2\times2$ square, namely $\{l-1,l\}\times\{\mu_{l-1}+1,\mu_{l-1}+2\}$,
contrary to the border strip conditions). Hence $\beta_{l}=\mu_{l-1}+1$
for all such $l$, and therefore $\beta_{i-j}=\mu_{i}+k-j$ by (\ref{eq:total_k}).
We must have $j\leq m$, as otherwise, we would have by (\ref{eq:mu_ineq}),
$\mu_{i-j}+(j-1)\geq\mu_{i-m-1}+m\geq\mu_{i}+k-1$, and consequently,
$\beta_{i-j}=\mu_{i}+k-j\leq\mu_{i-j}$, a contradiction. We have
therefore verified that $B_{\mu,i,k}\subseteq\{\lambda(0),..,\lambda(m)\}$.
In the other direction, it is straightforward to check that $\lambda(j)\in B_{\mu,i,k}$
for all $0\leq j\leq m$. This shows our claim.

Clearly, $\lambda(0),\dotsc,\lambda(m-1)$ are not Young diagrams,
as none of the moves is applicable on a Young diagram. We conclude
that $B_{\mu,i,k}\cap E_{\mu,k}\subseteq\{\lambda(m)\}=\{\delta\}$,
with equality if and only if $\delta$ is a Young diagram. This proves
(i). Finally, by equation (\ref{eq:j_th_step}), $\hght(\lambda(j)\backslash\mu)=j+1$
for all $j$. Setting $j=m$, we get $(ii)$.
\end{proof}
\begin{cor}
\label{cor:Ek-parametrization}Let $\mathbb{Y}$ be the set of all
Young diagrams. Then 
\[
E_{\mu,k}=\left\{ Y(F_{k,i}(\mu))|\,1\leq i\leq\hght(\mu)+k\right\} \cap\mathbb{Y}
\]
\end{cor}

We will use this parametrization of $E_{\mu,k}$ in the evaluation
of formula (\ref{eq:exp_formula}). As a first step, let us evaluate
the laplacian eigenvalue $\rho(Y(F_{k,i}(\mu))):$
\begin{lem}
\label{lem:rho_value}Assume that $Y(F_{k,i}(\mu))$ is a Young diagram,
where $n-k\vdash\mu$. We have
\[
\rho(Y(F_{k,i}(\mu)))=\binom{n}{2}+\sum_{(r,s)\in\mu}c((r,s))+k\left(i-\mu_{i}-\frac{k+1}{2}\right)
\]
\end{lem}

\begin{proof}
We observe that the sum $\sum_{(r,s)\in\lambda}c(r,s)$ is invariant
under the moves $R_{l}$ (as each such move is equivalent to moving
some boxes in the up-left direction, and $c((r,s))=c((r-1,s-1))$
for all $r,s$). Since $Y(F_{k,i}(\mu))$ is obtained from $F_{k,i}(\mu)$
by such moves, we can plug $F_{k,i}(\mu)=[\mu_{1},\dotsc,\mu_{i}+k,\dotsc,\mu_{r}]$
in formula (\ref{eq:character_ratio}) and obtain $\rho(Y(F_{k,i}(\mu)))$:

\[
\rho(Y(F_{k,i}(\mu)))=\binom{n}{2}+\sum_{(r,s)\in\mu}c((r,s))+\sum_{s=\mu_{i}+1}^{\mu_{i}+k}c((i,s))
\]
The result follows by evaluating the arithmetic sum in the last term.
\end{proof}

\section{\label{sec:The-hook-parametrization}The hook parametrization}
\begin{defn}
For a pre-diagram $\lambda=[\lambda_{1},\lambda_{2},\dotsc,\lambda_{r}]$
(where $\lambda_{i}\ge0$, $\lambda_{r}>0$) we call 
\[
\lambda^{(i)}=\lambda_{i}+r-i\,\,\,\,i=1,2,\dotsc,r
\]
\emph{the hook numbers} of $\lambda$.

If $\lambda$ is a Young diagram, then $\lambda^{(i)}$ is the length
of the hook whose corner is the box $(i,1)$. Also, for a Young diagram,
since $\lambda_{1}\geq\lambda_{2}\geq...\lambda_{r}\geq0$, the numbers
$\lambda^{(i)}$ satisfy $\lambda^{(1)}>\lambda^{(2)}>\dotsb>\lambda^{(r)}>0$,
and in particular, $\lambda^{(i)}\neq\lambda^{(j)}$ for $i\neq j$.
As we shall see, representing diagrams by their hook numbers leads
to a considerable simplification of formula (\ref{eq:exp_formula}).
\end{defn}

\begin{lem}
\label{thm:Hooks_and_Moves}Let $\lambda=[\lambda_{1},\dotsc,\lambda_{r}]$
be a pre-diagram, and let $\delta=Y(\lambda)$.

\begin{enumerate}
\item \label{enu:The-set-of}$\lambda$ and $\delta$ have the same hook
lengths (possibly in a different order).
\item $\delta$ is a Young diagram if and only if $\lambda^{(i)}\neq\lambda^{(j)}$
for all $i\neq j$.
\end{enumerate}
\end{lem}

\begin{proof}
~

\begin{enumerate}
\item Since $\delta$ is obtained from $\lambda$ by a sequence of the moves
$\{R_{j}\}$, it is enough to show that any such move preserves the
multiset of hook lengths. Indeed, the move $R_{j}$ replaces $\lambda_{j},\lambda_{j+1}$
by $\lambda_{j+1}-1,\lambda_{j}+1$, while leaving the other row lengths
intact. Thus the pair $\lambda^{(j)},\lambda^{(j+1)}$ is replaced
by $\lambda^{(j+1)},\lambda^{(j)}$, and the multiset $\{\lambda^{(1)},\dotsc,\lambda^{(r)}\}$
is preserved.
\item If $\delta$ is a Young diagram, then we have $\delta^{(i)}\neq\delta^{(j)}$
for $i\neq j$, and by (\ref{enu:The-set-of}), $\lambda^{(i)}\neq\lambda^{(j)}$
for all $i\neq j$. In the other direction, if $\lambda^{(i)}\neq\lambda^{(j)}$
for all $i\neq j$ then again by (\ref{enu:The-set-of}), we have
$\delta^{(i)}\neq\delta^{(j)}$ for $i\neq j$. Moreover, none of
the moves $\{R_{j}\}$ is applicable to $\delta$, so \textbf{$\delta_{j+1}\leq\delta_{j}+1$
}for all $1\leq j\leq r-1$. Since $\delta^{(j+1)}\neq\delta^{(j)}$
for any such $j$, we have $\delta_{j+1}\neq\delta_{j}+1$. Hence
$\delta_{1}\geq\delta_{2}\geq\dotsc\geq\delta_{r}$, and $\delta$
is a Young diagram.\qedhere
\end{enumerate}
\end{proof}
For any Young diagram $\lambda=[\lambda_{1},\dotsc,\lambda_{m}]$,
there is a well-known formula for the dimension $d_{\lambda}$ in
terms of the hook numbers $\lambda^{(i)}$:
\begin{lem}
(The Young-Frobenius formula). 
\[
d_{\lambda}=\frac{n!}{\prod_{t=1}^{m}\lambda^{(t)}!}\prod_{1\leq t<s\leq m}\left(\lambda^{(t)}-\lambda^{(s)}\right)
\]
\end{lem}

This formula was discovered by Frobenius \cite{F00} and independently
by Young \cite{Young}. It is not difficult to see its equivalence
to the more familiar hook formula. 

We can now evaluate the numbers $(-1)^{ht(\lambda\backslash\mu)+1}d_{\lambda}$
which appear in formula (\ref{eq:exp_formula}):
\begin{thm}
\label{thm:dim-formula}Let $\mu=[\mu_{1},\dotsc,\mu_{r}]$, $\lambda=F_{k,i}(\mu)$
and $\delta=Y(\lambda)$. Let $m=ht(\delta)=ht(\lambda)=\max\{i,r\}$.
We have
\begin{equation}
\frac{n!}{\prod_{t=1}^{m}\lambda^{(t)}!}\prod_{1\leq t<s\leq m}\left(\lambda^{(t)}-\lambda^{(s)}\right)=\begin{cases}
(-1)^{ht(\delta\backslash\mu)+1}d_{\delta} & \delta\mbox{ is a Young diagram}\\
0 & \mbox{otherwise}
\end{cases}\label{eq:signed_dim}
\end{equation}
 
\end{thm}

\begin{proof}
By lemma \ref{thm:Hooks_and_Moves}, $\delta$ is a Young diagram
if and only if all the numbers $\lambda^{(1)},\dotsc,\linebreak[0]\lambda^{(m)}$
are different from one another. Let us assume that $\delta$ is indeed
a Young diagram. Since $\{\delta^{(1)},\dotsc,\delta^{(m)}\}=\{\lambda^{(1)},\dotsc,\lambda^{(m)}\}$,
by the Young-Frobenius formula we have
\[
d_{\lambda}=\pm\frac{n!}{\prod_{t=1}^{m}\lambda^{(t)}!}\prod_{1\leq t<s\leq m}\left(\lambda^{(t)}-\lambda^{(s)}\right)
\]
The sign in this formula is the sign of the permutation which takes
$\{\delta^{(1)},\dotsc,\delta^{(m)}\}$ to $\{\lambda^{(1)},\dotsc,\lambda^{(m)}\}$.
We have seen in the proof of lemma \ref{thm:Hooks_and_Moves} that
each move $R_{j}$ applied to a pre-diagram induces a switch of two
consecutive values in the corresponding hook numbers. By lemma \ref{lem:wrap_process_2},
it takes $\hght(\lambda\backslash\mu)-1=\hght(\delta\backslash\mu)-1$
moves to get from $\lambda$ to $\delta$. This proves the formula.
\end{proof}

\section{The expectation of \texorpdfstring{$\alpha_{k}2^{\alpha}$}{alpha k times 2 to the alpha}
at time \texorpdfstring{$t$}{t}, continued\label{sec:tau}}

We are now ready to complete the calculation of $\mathbb{E}\left(\alpha_{k}\left(\pi(t)\right)2^{\alpha(\pi(t))}\right)$. 
\begin{thm}
\label{thm:onek formula}Let $y=e^{-tk}$. Then 
\[
\mathbb{E}\left(\alpha_{k}\left(\pi(t)\right)2^{\alpha(\pi(t))}\right)=\frac{2}{k}\sum_{n-k\vdash[a,b]}(a+1-b)\phi([a,b],k,t)
\]
where $\phi$ is given, for $b>0$, by
\begin{align*}
\phi([a,b],k,t) & =\binom{n}{k-1,a+1,b}\cdot e^{-t(ab+b)}\Biggl[\frac{(a+1-b)y^{a+b+2}(1-y)^{k-1}}{k}\\
 & \qquad\qquad+\;y^{b}(a+1-b+k)\int_{y}^{1}x^{a+1}(1-x)^{k-1}dx\\
 & \qquad\qquad+\;y^{a+1}(a+1-b-k)\int_{y}^{1}x^{b}(1-x)^{k-1}dx\Biggr]
\end{align*}
and for $b=0$ by
\[
\phi([n-k],k,t)=\binom{n}{k}\left[y^{n-k+1}(1-y)^{k-1}+k\int_{y}^{1}x^{n-k}(1-x)^{k-1}dx\right]
\]
\end{thm}

\begin{proof}
Let us repeat equation (\ref{eq:exp_formula}):

\[
\mathbb{E}\left(\alpha_{k}\left(\pi(t)\right)2^{\alpha(\pi(t))}\right)=\frac{2}{k}\sum_{n-k\vdash\mu=[a,b]}(a-b+1)\sum_{\lambda\in E_{k}(\mu)}(-1)^{\hght(\lambda\backslash\mu)+1}d_{\lambda}e^{-t\rho(\lambda)}
\]
where $E_{k}(\mu)$ is the set of $\lambda\dashv n$ such that $\mu\subset\lambda$
and $\lambda\setminus\mu$ is a border strip of size $k$. 

We now fix some $n-k\vdash\mu=[a,b]$, and denote the inner sum in
the above equation by
\[
\tau(\mu,k,t)=\sum_{\lambda\in E_{k}(\mu)}(-1)^{ht(\lambda\backslash\mu)+1}d_{\lambda}e^{-t\rho(\lambda)}
\]
Our goal is to prove that $\tau(\mu,k,t)=\phi(\mu,k,t)$. We treat
the case $b>0$ first.

By corollary \ref{cor:Ek-parametrization} and theorem \ref{thm:dim-formula},
we may write $\tau(\mu,k,t)$ as
\begin{equation}
\tau(\mu,k,t)=\sum_{i=1}^{k+2}D(\mu,k,i)\exp\big(-t\rho(Y(F_{k,i}(\mu)))\big)\label{eq:tau_def}
\end{equation}
where

\begin{align*}
D(\mu,k,i) & =\begin{cases}
(-1)^{ht(Y(F_{k,i}(\mu)\backslash\mu))+1}d_{Y(F_{k,i}(\mu))} & Y(F_{k,i}(\mu))\mbox{ is a Young diagram}\\
0 & \mbox{otherwise}
\end{cases}\\
 & =\frac{n!}{\prod_{t=1}^{m}\gamma^{(t)}!}\prod_{1\leq t<s\leq m}\left(\gamma^{(t)}-\gamma^{(s)}\right)
\end{align*}
and where $\gamma=[\gamma_{1},\dotsc,\gamma_{m}]=F_{k,i}(\mu)$. Note
that $\rho(Y(F_{k,i}(\mu)))$ was not defined when $Y(F_{k,i}(\mu))$
is not a Young diagram \textemdash{} for some values of $i$ and $k$
it may just be a prediagram \textemdash{} but in such cases, it is
multiplied by zero in (\ref{eq:tau_def}), so we may define it arbitrarily.
It turns out that the cases $i=1$ and $2$ are best handled separately,
so write 
\[
\tau(\mu,k,t)=\Romanbar{I}+\Romanbar{II}+\Romanbar{III}
\]
 where $\Romanbar{I}=D(\mu,k,1)\exp(-t\rho(Y(F_{k,1}(\mu))))$, $\Romanbar{II}=D(\mu,k,2)\exp(-t\rho(Y(F_{k,2}(\mu))))$
and $\Romanbar{III}$ is the sum over the remaining terms.

Let us consider a few cases and write $\gamma$ and its vector of
hook lengths, $(\gamma^{(1)},\dotsc,\linebreak[0]\gamma^{(m)})$ in
each case:

\[
\gamma=\begin{cases}
[a+k,b] & i=1\\{}
[a,b+k] & i=2\\{}
[a,b,0^{i-3},k] & i>2
\end{cases}
\]
and consequently
\[
(\gamma^{(1)},\dotsc,\gamma^{(m)})=\begin{cases}
(a+k+1,b) & i=1\\
(a+1,b+k) & i=2\\
(a+i-1,b+i-2,i-3,i-4,\dotsc,1,k) & i>2
\end{cases}
\]

We conclude:
\begin{align}
D(\mu,k,i) & =\begin{cases}
\frac{n!(a+k+1-b)}{(a+k+1)!b!} & i=1\\
\frac{n!(a+1-b-k)}{(a+1)!(b+k)!} & i=2
\end{cases}\label{eq:D_for_i=00003D1,2}
\end{align}
while for $i>2$ the main term is 
\begin{align*}
\prod_{1\leq t<s\leq m}\left(\gamma^{(t)}-\gamma^{(s)}\right) & =(a+1-b)\left(\prod_{j=a+2}^{a+i-2}j\right)(a+i-1-k)\left(\prod_{j=b+1}^{b+i-3}j\right)\\
 & \qquad\cdot(b+i-2-k)\left(\prod_{j=1}^{i-4}j!\right)\left(\prod_{j=1-k}^{i-3-k}j\right)\\
 & =(a+1-b)(a+i-1-k)(b+i-2-k)\\
 & \frac{(a+i-2)!}{(a+1)!}\frac{(b+i-3)!}{b!}\left(\prod_{j=1}^{i-4}j!\right)(-1)^{i-3}\frac{(k-1)!}{(k+2-i)!}
\end{align*}
and the term in the denominator is
\[
\prod_{t=1}^{m}\gamma^{(t)}!=(a+i-1)!(b+i-2)!\left(\prod_{j=1}^{i-3}j!\right)k!
\]
Hence, after some simplification,
\begin{multline}
D(\mu,k,i)=\\
\frac{a+1-b}{k}\binom{n}{k-1,a+1,b}(-1)^{i+1}\binom{k-1}{i-3}\frac{(a+i-1-k)(b+i-2-k)}{(a+i-1)(b+i-2)}.\label{eq:summand_formula}
\end{multline}
By a partial fraction decomposition,
\begin{align*}
\lefteqn{\frac{(a+i-1-k)(b+i-2-k)}{(a+i-1)(b+i-2)}=}\qquad\qquad\\
 & =\left(1-\frac{k}{a+i-1}\right)\left(1-\frac{k}{b+i-2}\right)\\
 & =1-\frac{k}{a+i-1}-\frac{k}{b+i-2}+\frac{k^{2}}{\left(a+i-1\right)\left(b+i-2\right)}\\
 & =1-\frac{k}{a+i-1}-\frac{k}{b+i-2}-\frac{k^{2}}{a+1-b}\left(\frac{1}{a+i-1}-\frac{1}{b+i-2}\right)\\
 & =\frac{1}{a+1-b}\left((a+1-b)-\frac{k(a+1-b+k)}{a+i-1}-\frac{k(a+1-b-k)}{b+i-2}\right)
\end{align*}
so 
\begin{align*}
D(\mu,k,i) & =\frac{1}{k}\binom{n}{k-1,a+1,b}(-1)^{i+1}\binom{k-1}{i-3}\\
 & \cdot\left((a+1-b)-\frac{k(a+1-b+k)}{a+i-1}-\frac{k(a+1-b-k)}{b+i-2}\right)
\end{align*}
Next, let us calculate the eigenvalue $\rho(Y(F_{k,i}(\mu))).$ We
have
\[
\sum_{(i,j)\in\mu=[a,b]}c((i,j))=-\frac{a(a-1)}{2}-\frac{(b-1)(b-2)}{2}+1
\]
and by lemma \ref{lem:rho_value},
\begin{align}
\rho(Y(F_{k,i}(\mu))) & =\binom{n}{2}+\sum_{(i,j)\in\mu=[a,b]}c((i,j))+k\left(i-\mu_{i}-\frac{k+1}{2}\right)\nonumber \\
 & =\frac{n(n-1)}{2}-\frac{a(a-1)}{2}-\frac{(b-1)(b-2)}{2}-\frac{k(k+1)}{2}+1+ki-k\mu_{i}\nonumber \\
 & =ki-k\mu_{i}+\frac{1}{2}\Big(2+(a+b+k)(a+b+k-1)\nonumber \\
 & \qquad\qquad-\;a(a-1)-(b-1)(b-2)-k(k+1)\Big)\nonumber \\
 & =ki-k\mu_{i}+ab+ak+bk+b-k.\label{eq:general_rho_f}
\end{align}
For $i>2$, we have $\mu_{i}=0$, and we conclude that
\[
\exp\big(-t\rho(Y(F_{k,i}(\mu)))\big)=e^{-tki}e^{-t\left(ab+ak+bk+b-k\right)}
\]
Let us put everything together:
\begin{align*}
\Romanbar{III} & =\sum_{i=3}^{k+2}D(\mu,k,i)\exp\big(-t\rho(Y(F_{k,i}(\mu)))\big)\\
 & =e^{-t\left(ab+ak+bk+b-k\right)}\frac{1}{k}\binom{n}{k-1,a+1,b}\cdot\\
 & \cdot\sum_{i=3}^{k+2}(-1)^{i+1}\binom{k-1}{i-3}e^{-tki}\left((a+1-b)-\frac{k(a+1-b+k)}{a+i-1}-\frac{k(a+1-b-k)}{b+i-2}\right)
\end{align*}
Let us move $e^{-3tk}$ outside, and replace $i-3$ with $i$:
\begin{multline*}
\Romanbar{III}=e^{-t\left(ab+ak+bk+b-k\right)}\frac{1}{k}\binom{n}{k-1,a+1,b}\cdot e^{-3tk}\\
\cdot\sum_{i=0}^{k-1}(-1)^{i}\binom{k-1}{i}e^{-tki}\left((a-b+1)-\frac{k(a+1-b+k)}{a+i+2}-\frac{k(a+1-b-k)}{b+i+1}\right)
\end{multline*}
Let us evaluate the inner sum. For the first summand we use the binomial
formula:
\[
\sum_{i=0}^{k-1}(-1)^{i}\binom{k-1}{i}e^{-tki}=(1-e^{-tk})^{k-1}
\]
For the other two terms, we use integration and the binomial formula:

\begin{align*}
\sum_{i=0}^{k-1}(-1)^{i}\binom{k-1}{i}\frac{e^{-tki}}{a+i+2} & =e^{tk(a+2)}\sum_{i=0}^{k-1}(-1)^{i}\binom{k-1}{i}\frac{e^{-tk(a+i+2)}}{a+i+2}\\
 & =e^{tk(a+2)}\sum_{i=0}^{k-1}(-1)^{i}\binom{k-1}{i}\int_{0}^{e^{-tk}}x^{a+i+1}\,dx\\
 & =e^{tk(a+2)}\int_{0}^{e^{-tk}}x^{a+1}\sum_{i=0}^{k-1}\binom{k-1}{i}(-x)^{i}\,dx\\
 & =e^{tk(a+2)}\int_{0}^{e^{-tk}}x^{a+1}(1-x)^{k-1}\,dx
\end{align*}
and similarly,
\[
\sum_{i=0}^{k-1}(-1)^{i}\binom{k-1}{i}\frac{e^{-tki}}{b+i+1}=e^{tk(b+1)}\int_{0}^{e^{-tk}}x^{b}(1-x)^{k-1}\,dx
\]
Collecting the 3 terms, we obtain
\begin{align*}
\Romanbar{III} & =e^{-t\left(ab+ak+bk+b+2k\right)}\frac{1}{k}\binom{n}{k-1,a+1,b}\cdot\Bigg[(a+1-b)(1-e^{-tk})^{k-1}\\
 & \qquad-k(a+1-b+k)e^{tk(a+2)}\int_{0}^{e^{-tk}}x^{a+1}(1-x)^{k-1}\,dx\\
 & \qquad-k(a+1-b-k)e^{tk(b+1)}\int_{0}^{e^{-tk}}x^{b}(1-x)^{k-1}\,dx\bigg]
\end{align*}
and with the notation $y=e^{-tk}$:
\begin{align*}
\Romanbar{III} & =e^{-t\left(ab+b\right)}\binom{n}{k-1,a+1,b}\cdot\Bigg[\frac{a+1-b}{k}(1-y)^{k-1}y^{a+b+2}\\
 & \qquad-(a+1-b+k)y^{b}\int_{0}^{y}x^{a+1}(1-x)^{k-1}\,dx\\
 & \qquad-(a+1-b-k)y^{a+1}\int_{0}^{y}x^{b}(1-x)^{k-1}\,dx\bigg].
\end{align*}
Let us now add the terms $\Romanbar{I}$ and $\Romanbar{II}$. By
(\ref{eq:D_for_i=00003D1,2}) and (\ref{eq:general_rho_f}),
\begin{align*}
\Romanbar{I} & =D(\mu,k,1)\exp(-t\rho(Y(F_{k,1}(\mu))))\\
 & =\frac{n!(a+k+1-b)}{(a+k+1)!b!}e^{-t(ab+bk+b)}\\
 & =\binom{n}{k-1,a+1,b}\frac{(a+1)!(k-1)!(a+k+1-b)}{(a+k+1)!}e^{-t(ab+b)}y^{b}\\
\intertext{and}\Romanbar{II} & =D(\mu,k,2)\exp(-t\rho(Y(F_{k,2}(\mu))))\\
 & =\frac{n!(a+1-b-k)}{(a+1)!(b+k)!}e^{-t(k+ab+ak+b)}\\
 & =\binom{n}{k-1,a+1,b}\frac{b!(k-1)!(a+1-b-k)}{(b+k)!}e^{-t(ab+b)}y^{a+1}
\end{align*}
Collecting terms, we get the following formula for the inner sum
in equation (\ref{eq:exp_formula}) (see also equation (\ref{eq:tau_def})):
\begin{align}
\tau(\mu,k,t) & =\binom{n}{k-1,a+1,b}e^{-t(ab+b)}\cdot\Bigg[\frac{(a+1-b)}{k}y^{a+b+2}(1-y)^{k-1}\label{eq:five-term-tau}\\
 & \qquad-(a+1-b+k)y^{b}\int_{0}^{y}x^{a+1}(1-x)^{k-1}\,dx\nonumber \\
 & \qquad-(a+1-b-k)y^{a+1}\int_{0}^{y}x^{b}(1-x)^{k-1}\,dx\nonumber \\
 & \qquad+\frac{(a+1)!(k-1)!(a+k+1-b)}{(a+k+1)!}y^{b}\nonumber \\
 & \qquad+\frac{b!(k-1)!(a+1-b-k)}{(b+k)!}y^{a+1}\bigg].\nonumber 
\end{align}
We have the following Beta function identity:
\[
\int_{0}^{1}x^{a+1}(1-x)^{k-1}dx=B(a+2,k)=\frac{(k-1)!(a+1)!}{(a+k+1)!}
\]
Hence we may group the second and fourth summands in (\ref{eq:five-term-tau}):
\begin{align*}
\lefteqn{-(a+1-b+k)y^{b}\int_{0}^{y}x^{a+1}(1-x)^{k-1}\,dx+\frac{(a+1)!(k-1)!(a+k+1-b)}{(a+k+1)!}y^{b}}\qquad\qquad\\
 & =(a+1-b+k)y^{b}\left(\int_{0}^{1}x^{a+1}(1-x)^{k-1}\,dx-\int_{0}^{y}x^{a+1}(1-x)^{k-1}\,dx\right)\\
 & =(a+1-b+k)y^{b}\int_{y}^{1}x^{a+1}(1-x)^{k-1}\,dx.\\
\intertext{\text{Similarly, for the third and fifth terms, we have}}\lefteqn{-(a+1-b-k)y^{a+1}\int_{0}^{y}x^{b}(1-x)^{k-1}\,dx+\frac{b!(k-1)!(a+1-b-k)}{(b+k)!}y^{a+1}}\qquad\qquad\\
 & =(a+1-b-k)y^{a+1}\int_{y}^{1}x^{b}(1-x)^{k-1}\,dx.
\end{align*}
Therefore, (\ref{eq:five-term-tau}) becomes
\begin{align*}
\tau(\mu,k,t) & =\binom{n}{k-1,a+1,b}\cdot e^{-t(ab+b)}\Biggl[\frac{(a+1-b)y^{a+b+2}(1-y)^{k-1}}{k}\\
 & \qquad\qquad+\;y^{b}(a+1-b+k)\int_{y}^{1}x^{a+1}(1-x)^{k-1}dx\\
 & \qquad\qquad+\;y^{a+1}(a+1-b-k)\int_{y}^{1}x^{b}(1-x)^{k-1}dx\Biggr]=\phi(\mu,k,t)
\end{align*}
as desired.

Finally, let us treat the case $b=0$, in which $\mu=[n-k]$. We will
see that this case is covered by the calculation in \cite{BK}. By
\cite[lemma 3]{AK}, 
\begin{equation}
\ch(\alpha_{k})=\frac{1}{k}u_{k}S_{[n-k]}\label{eq:ch-of-alpha}
\end{equation}
by (\ref{eq:StanleyFormula}),
\begin{equation}
\ch(\alpha_{k})=\frac{1}{k}\sum_{\lambda\in E_{k}(\mu)}(-1)^{\hght(\lambda\backslash\mu)+1}S_{\lambda}\label{eq:ch-of-alpha-2}
\end{equation}
which is equivalent to 
\begin{equation}
\alpha_{k}=\frac{1}{k}\sum_{\lambda\in E_{k}(\mu)}(-1)^{\hght(\lambda\backslash\mu)+1}\chi_{\lambda}.\label{eq:alpha_decomp}
\end{equation}
By lemma \ref{lem:exp_char},
\begin{align*}
\mathbb{E}(\alpha_{k}(\pi(t))) & =\frac{1}{k}\sum_{\lambda\in E_{k}(\mu)}(-1)^{ht(\lambda\backslash\mu)+1}d_{\lambda}e^{-t\rho(\lambda)}\\
 & =\frac{1}{k}\tau(\mu,k,t)
\end{align*}
On the other hand, $\mathbb{E}(\alpha_{k}(\pi(t))$ has been calculated
in \cite[Theorem 1]{BK}:

\[
\mathbb{E}(\alpha_{k}(\pi(t)))=\binom{n}{k}\left[\frac{1}{k}y^{n-k+1}(1-y)^{k-1}+\int_{y}^{1}x^{n-k}(1-x)^{k-1}dx\right]
\]
comparing the two formulas for $\mathbb{E}(\alpha_{k}(\pi(t)))$,
we conclude that $\tau(\mu,k,t)=\phi(\mu,k,t)$ for $b=0$ as well.
\end{proof}

\begin{proof}
[Proof of theorem \ref{thm:precise}]Theorem \ref{thm:precise} is
a direct consequence of theorem \ref{thm:onek formula}, the definition
of $m^{2}(t)$ and $Z$ (equation (\ref{eq:Zmalpha})), the fact that
$\psi(b,k)=(a+1-b)\phi([n-k-b,b],k,t)$ ($\psi$ from theorem \ref{thm:precise}
and $\phi$ from theorem \ref{thm:onek formula}) and the fact that
$\sum k\alpha_{k}=n$ so $Z=\mathbb{E}2^{\alpha}=\frac{1}{n}\sum k\mathbb{E}(\alpha_{k}2^{\alpha})$.
\end{proof}

\section{Proof of theorem \ref{thm:phase transition}}

The proof of theorem \ref{thm:phase transition} is nothing but estimating
carefully the terms in theorem \ref{thm:precise} and is pretty straightforward,
but let us anyway start with a bird's eye view. Recall the terms $\psi(b,k,t)$
from theorem \ref{thm:precise} and that 
\[
m_{n}^{2}(t)=n\frac{\sum_{b,k}k\psi(b,k,t)}{\sum_{b,k}\psi(b,k,t)}.
\]
where the sums are over all ranges of the parameters i.e.\ $k$
runs from $1$ to $n$ and $b$ runs from $0$ to $\left\lfloor (n-k)/2\right\rfloor $.
We will find the set of $(b,k)$ for which $\psi$ is maximal (depending
on $t$) and the rest will turn out to be negligible. For visualisation
purposes it is convenient to define normalised values $\tau=tn$,
$\kappa=k/n$ and $\beta=b/n$ (no relation to the $\beta$ from \S \ref{subsec:quantum}
or the $\tau$ from \S \ref{sec:tau}). The relation $b\le\frac{1}{2}(n-k)$
translates to $\beta\le\frac{1}{2}(1-\kappa)$ (see the diagonal in
figure \ref{fig:betatau}). 
\begin{figure}
\input{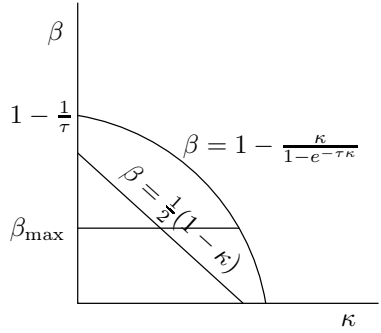}\caption{\label{fig:betatau}$\beta$ and $\tau$ (the case $\tau>2$ depicted)}
\end{figure}
We return therefore to the definition of $\psi$ in theorem \ref{thm:precise}
and recall that it has three terms (below we abbreviate $\psi(b,k)=\psi(b,k,t)$).
It will turn out that the term containing $\int_{y}^{1}x^{a+1}(1-x)^{k-1}$
(recall that $a=n-k-b$, that $a\ge b$ and that $y=e^{-tk}$) is
the most significant one. The integrand has its maximum in the interval
$[0,1]$ at $(a+1)/(a+k)$ and, as we will see, cases where the integral
does not capture this maximum, i.e.\ that $y>(a+1)/(a+k),$ are also
negligible. Thus we may restrict our attention to the corresponding
part of the $\kappa$, $\beta$ plane. Since $y=e^{-\tau\kappa}$
and 
\[
\frac{a+1}{a+k}=\frac{n-k-b+1}{n-b}=\frac{1-\kappa-\beta}{1-\beta}+o(1)
\]
the condition becomes $\beta\le1-\kappa/(1-e^{-\tau\kappa})$ (see
figure \ref{fig:betatau}, and especially note the intersection with
the $\kappa=0$ line, $\beta=1-1/\tau$). Further, when the integral
does capture the maximum, it may be estimated very well by the integral
from $0$ to $1$, which is a Beta integral and has an explicit formula.
We get for such $b$ and $k$,
\begin{multline*}
\psi(b,k)\sim(a+1-b)\binom{n}{k-1,a+1,b}e^{-t(ab+b)}y^{b}(a+1-b+k)\int_{0}^{1}x^{a+1}(1-x)^{k-1}\\
=\binom{n}{b}e^{-tb(n-b+1)}(n-k-2b+1)\frac{n+1-2b}{n-k}.
\end{multline*}
The right-hand side depends rather weakly on $k$. Dropping some polynomial
terms (in particular all terms containing $k$) we may write
\begin{align*}
\psi(b,k) & \sim\binom{n}{b}e^{-tb(n-b)}\\
 & \sim\left(\frac{n}{b}\right)^{b}\left(\frac{n}{n-b}\right)^{n-b}e^{-tb(n-b)}=\Big(\beta^{-\beta}(1-\beta)^{-(1-\beta)}e^{-\tau\beta(1-\beta)}\Big)^{n}
\end{align*}
We see that, as a function of $\beta$, its maximum can be found by
differentiating the expression on the right, and then solving a transcendental
equation. Denote the $\beta$ at which the maximum is attained by
$\beta_{\max}$. We do not care about the exact value of $\beta_{\max}$,
we only need to know whether the line $\{\beta=\beta_{\max}\}$ intersects
the region below both the curve and the diagonal in figure \ref{fig:betatau}
(they are not always arranged as in the picture, this depends on the
value of $\tau$), and for this we need only check the points $(\beta\in\{1-1/\tau,\nicefrac{1}{2}\},\kappa=0)$.
A simple check shows that $\beta_{\max}<\min(\frac{1}{2},1-\frac{1}{\tau})$
happens exactly when $\tau>2$, and hence this is the critical value
for the appearance of large cycles: when $\tau<2$ the ``mass''
sits at the point $(\beta=\max(0,1-1/\tau),\kappa=0)$ and in particular
there are no large cycles. When $\tau>2$ there is approximately equal
mass at each $k$ for which $\beta_{\max}\le1-\kappa/(1-e^{-\tau\kappa})$
and in particular there are macroscopic cycles (very small $k$, of
constant order, have larger mass, reflecting the fact that a positive
proportion of the mass still sits in small cycles). The details of
the calculation are below.
\begin{proof}
We start with the case of $\tau>2$ for which we need to show $m\apprge n$.
Fix some $\tau>2$. Let $\mu$ be the function $\mu(\beta)=(1-\beta)\log(1-\beta)+\beta\log\beta+\tau\beta(1-\beta)$,
let $\beta_{max}$ be the unique minimiser of $\mu$ in the interval
$[0,\tfrac{1}{2}]$, and note that $\beta_{\max}<\frac{1}{2}$ (here
and below we skip the details of the calculus exercises involved).
 Let $\kappa_{0}>0$ satisfy $\beta_{\max}<1-\kappa_{0}/(1-e^{-\tau\kappa_{0}})$
(note that $1-\kappa/(1-e^{-\tau\kappa})$ is a decreasing function
of $\kappa$, so that the inequality $\beta_{\max}<1-\kappa/(1-e^{-\tau\kappa})$
holds for all $\kappa<\kappa_{0}$). Assume in addition that $\kappa_{0}<\frac{1}{2}-\beta_{\max}$.
We will estimate, for $k<\kappa_{0}n$, the contribution of the term
$\sum_{b}k\psi(b,k)/\sum_{l,b}\psi(b,l)$ at time $t=\tau/n$ and
this will establish the theorem in the case $\tau>2$. Hence we need
to bound $\sum_{b,l}\psi(b,l)$ from above and $\sum_{b}\psi(b,k)$
from below.

For the lower bound we first note that $\psi(b,k)\ge0$ for all $b$
and $k$. Indeed, recall that the definition of $\psi$ (page \pageref{thm:precise})
has three summands, two (the second and third) involving partial beta
integrals and one not involving any integral. The only summand which
might be negative is the third, because of the term $a+1-b-k$ which
is sometimes negative (recall that $a=n-k-b$ and $y=e^{-tk}$). But
the second summand cancels it. Indeed, because $a+1>b$,
\begin{equation}
y^{a+1}\int_{y}^{1}x^{b}(1-x)^{k-1}\le y^{b}\int_{y}^{1}x^{a+1}(1-x)^{k-1}\label{eq:useful}
\end{equation}
and the sum of the two integrals in the definition of $\psi$ is
bigger or equal to $2(a+1-b)y^{b}\int_{y}^{1}x^{a+1}(1-x)^{k-1}$
and in particular is positive.

Therefore, to prove a lower bound we may restrict our attention to
only a subset of the allowed $b$ and $k$. We take $b$ such that
$|b-n\beta_{\max}|\le2\sqrt{n}$ and $k\in(\frac{1}{2}n\kappa_{0},n\kappa_{0})$.
Next we drop the first and third summands in the definition of $\psi$,
which we may as they are positive (the term $a+1-b-k$ is positive
for our values of $b$ and $k$ because of our assumption that $\kappa_{0}<\frac{1}{2}-\beta_{\max}$,
if $n$ is sufficiently large). We get
\[
\psi(b,k)\ge(a+1-b)\frac{n!}{(k-1)!(a+1)!b!}e^{-t(ab+b)}y^{b}(a+1-b+k)\int_{y}^{1}x^{a+1}(1-x)^{k-1}\,dx.
\]
Since $\beta_{\max}<\frac{1}{2}-\kappa_{0}$ we have 
\begin{equation}
a+1-b\apprge n\label{eq:nu beemet}
\end{equation}
 (here and throughout the proof, the implicit constant in the notation
$\apprge$ may depend only on $\tau$ and $\kappa_{0}$). Further,
we claim that
\begin{equation}
\int_{y}^{1}x^{a+1}(1-x)^{k-1}\,dx\apprge\int_{0}^{1}x^{a+1}(1-x)^{k-1}.\label{eq:no y}
\end{equation}
To see (\ref{eq:no y}) we first compare the integral to the maximum
of its integrand. On the one hand, the integral from $0$ to $1$
is equal to $(k-1)!(a+1)!/(a+k+1)!$. On the other hand the maximum
of the integrand is at $x_{\max}=(a+1)/(a+k)$. With Stirling's formula
we get
\begin{equation}
\int_{0}^{1}x^{a+1}(1-x)^{k-1}\approx\frac{\sqrt{k}}{n}x_{\max}^{a+1}(1-x_{\max})^{k-1}\label{eq:integral max integrand}
\end{equation}
where the notation $\approx$ means that the ratio between the two
sides is bounded above and below by constants depending only on $\tau$
and $\kappa_{0}$ (the estimate (\ref{eq:integral max integrand})
holds for any $b$ and $k$, not just under the restrictions above).

Returning to (\ref{eq:no y}), we note that 
\[
x_{\max}=1-\frac{k-1}{n-b}=1-\frac{\kappa}{1-\beta_{\textrm{max}}}+O(n^{-1/2}).
\]
Our choice of parameters gives $e^{-\tau\kappa}<1-\kappa/(1-\beta_{\max})$,
and this holds uniformly in $\kappa$, i.e.
\[
y=e^{-\tau\kappa}<1-\frac{\kappa}{1-\beta_{\textrm{max}}}+c(\tau,\kappa_{0})=x_{\textrm{max}}+c(\tau,\kappa_{0})+O(n^{-1/2})
\]
for some $c(\tau,\kappa_{0})>0$. Returning to $x^{a+1}(1-x)^{k-1}$
it is easy to check that it decays exponentially away from its maximum.
All these give us that 
\[
\int_{0}^{y}x^{a+1}(1-x)^{k-1}\le y^{a+1}(1-y)^{k-1}<(1-c_{c}(\tau,\kappa_{0}))^{n}x_{\max}^{a+1}(1-x_{\max})^{k-1}.
\]
The exponential factor $(1-c_{2})^{n}$ wins over the factor $\sqrt{k}/n$
in (\ref{eq:integral max integrand}), so we get $\int_{0}^{y}\ll\int_{0}^{1}$.
This establishes (\ref{eq:no y}) for $n$ sufficiently large, and
hence for every $n$ (perhaps with a different value for the constant
implicit in the notation $\apprge$).

We get
\begin{align*}
\psi(b,k) & \ge(a+1-b)\frac{n!}{(k-1)!(a+1)!b!}e^{-t(ab+b)}y^{b}(a+1-b+k)\;\cdot\\
 & \qquad\qquad\cdot\;\int_{y}^{1}x^{a+1}(1-x)^{k-1}\,dx\\
 & \stackrel{\clap{\ensuremath{{\scriptstyle {\textrm{{(\ref{eq:nu beemet},\ref{eq:no y})}}}}}}}{\gtrsim}n\frac{n!}{(k-1)!(a+1)!b!}e^{-tb(a+1)}e^{-tkb}\cdot n\cdot\int_{0}^{1}x^{a+1}(1-x)^{k-1}\,dx\\
 & =n^{2}\frac{n!}{(k-1)!(a+1)!b!}e^{-tb(a+1+k)}\frac{(a+1)!(k-1)!}{(a+k+1)!}\\
 & =n^{2}\frac{n!}{b!(n-b+1)!}e^{-tb(n-b+1)}\gtrsim n\binom{n}{b}e^{-tb(n-b)}
\end{align*}
where in the last inequality we used that $e^{-tb}\approx1$ and that
$n-b\approx n$. This holds for all $b$ satisfying $|b-n\beta_{\max}|\le2\sqrt{n}$,
and for such $b$ we have, by Stirling's formula,
\begin{align}
\binom{n}{b}e^{-tb(n-b)} & \approx\frac{1}{\sqrt{n}}\exp\left(-\Big(b\log\frac{b}{n}+(n-b)\log\frac{n-b}{n}+tb(n-b)\Big)\right)\nonumber \\
 & =\frac{1}{\sqrt{n}}\exp\left(-n\mu\Big(\frac{b}{n}\Big)\right)\apprge\frac{1}{\sqrt{n}}\exp\left(-n\mu(\beta_{\max})\right)\label{eq:enter the mu}
\end{align}
where the second inequality comes from Taylor expanding $\mu$ near
$\beta_{\max}$ to second order (since $\beta_{\max}$ is the minimum
of $\mu$, $\mu'(\beta_{\max})=0$). Summing over $b$ thus gives
\begin{equation}
\sum_{b}\psi(b,k)\apprge n\exp(-n\mu(\beta_{\max})).\label{eq:lower}
\end{equation}
This ends the lower bound.

For the upper bound we need to consider all 3 summands in the definition
of $\psi$ in theorem \ref{thm:precise}, as well as the case $b=0$.
The easiest are the integrals (the second and third summands). For
the second we write
\begin{equation}
y^{b}\int_{y}^{1}x^{a+1}(1-x)^{k-1}\,dx<y^{b}\int_{0}^{1}x^{a+1}(1-x)^{k-1}\,dx.\label{eq:term II}
\end{equation}
For the third, (\ref{eq:useful}) gives the same estimate:
\begin{equation}
y^{a+1}\int_{y}^{1}x^{b}(1-x)^{k-1}\,dx\stackrel{\textrm{(\ref{eq:useful},\ref{eq:term II})}}{<}y^{b}\int_{0}^{1}x^{a+1}(1-x)^{k-1}.\label{eq:termIII}
\end{equation}
For the first summand we use (\ref{eq:integral max integrand}) to
write
\begin{equation}
y^{a+b+2}(1-y)^{k-1}\le x_{\max}^{a+1}(1-x_{\max})^{k-1}y^{b+1}\stackrel{\clap{\ensuremath{{\scriptstyle \textrm{(\ref{eq:integral max integrand})}}}}}{\apprle}\frac{n}{\sqrt{k}}y^{b}\int_{0}^{1}x^{a+1}(1-x)^{k-1}.\label{eq:termI}
\end{equation}
(the divergence as $k\to1$ reflects the fact that indeed, in our
time scale, there are still many small cycles). We covered the three
terms in theorem \ref{thm:precise} in (\ref{eq:term II}), (\ref{eq:termIII})
and (\ref{eq:termI}) so we may write
\begin{align}
\psi(b,k,t) & \apprle\binom{n}{k-1,a+1,b}e^{-tb(a+1)}\cdot n^{2}(1+nk^{-3/2})y^{b}\int_{0}^{1}x^{a+1}(1-x)^{k-1}\nonumber \\
 & =\frac{n!}{(n-b+1)!b!}n^{2}(1+nk^{-3/2})e^{-tb(n-b-1)}\nonumber \\
 & \apprle n\binom{n}{b}e^{-tb(n-b)}(1+nk^{-3/2}).\label{eq:upper one b}
\end{align}
Formally, we only showed (\ref{eq:upper one b}) for $b>0$, but an
inspection of the case $b=0$ shows that it differs from that $b>0$
case only in details of the polynomial factors and in terms such as
$y^{C}$, all of which may be ``folded'' into the implicit constant.
Hence (\ref{eq:upper one b}) holds also for $b=0$.

All that remains is to sum (\ref{eq:upper one b}) over all $k$ and
$b$. Summing over $k$ simply gives another factor of $n$ i.e.\ $n^{2}\binom{n}{b}e^{-tb(n-b)}$.
As for the sum over $b$, a simple calculus exercise shows that our
function $\mu$ defined by $\mu(\beta)=\beta\log\beta+(1-\beta)\log(1-\beta)+\tau\beta(1-\beta)$
satisfies 
\begin{equation}
\mu(\beta)\ge\mu(\beta_{\max})+c(\beta-\beta_{\max})^{2}\label{eq:mu quadratic}
\end{equation}
for some constant $c$ depending only on $\tau.$ Hence we have
\begin{multline*}
\sum_{b=0}^{n/2}\binom{n}{b}e^{-tb(n-b)}\apprle\sum_{b=0}^{n/2}\frac{1}{\sqrt{n}}\exp\left(-n\mu\Big(\frac{b}{n}\Big)\right)\\
\apprle\frac{1}{\sqrt{n}}\exp(-n\mu(\beta_{\max}))\sum_{b=0}^{n/2}\exp\left(-nc\Big(\frac{b}{n}-\beta_{\max}\Big)^{2}\right)\apprle\exp(-n\mu(\beta_{\max}))
\end{multline*}
(we skip the details of the last calculation, which is standard).
All in all we get
\begin{equation}
\sum_{k,b}\psi(b,k)\apprle n^{2}\exp(-n\mu(\beta_{\max})).\label{eq:upper}
\end{equation}
With (\ref{eq:lower}) the case of $\tau>2$ is finished: we get,
for every $k\in(\frac{1}{2}n\kappa_{0},n\kappa_{0})$
\[
\frac{\sum_{b}\psi(b,k)}{\sum_{b,l}\psi(b,l)}\apprge\frac{n\exp(-n\mu(\beta_{\max}))}{n^{2}\exp(-n\mu(\beta_{\max}))}=\frac{1}{n}.
\]
Hence
\[
m^{2}=n\sum_{k=0}^{n}k\frac{\sum_{b}\psi(b,k)}{\sum_{b,l}\psi(b,l)}\apprge n\sum_{k=\lceil\frac{1}{2}n\kappa_{0}\rceil}^{\lfloor n\kappa_{0}\rfloor}\frac{k}{n}\apprge n^{2}
\]
as needed.

\subsection*{The case \texorpdfstring{$\tau<2$}{tau smaller than 2}.}

We retain the notation $\mu(\beta)=\beta\log\beta+(1-\beta)\log(1-\beta)+\tau\beta(1-\beta)$
from the previous part, but note that in this case this function is
decreasing on all of $[0,\frac{1}{2}]$. Again we need upper and lower
bounds on $\psi(b,k)$.

The lower bound follows by considering $k=1$ and $b\in[\frac{1}{2}n-2\sqrt{n},\frac{1}{2}n-\sqrt{n}]$,
and by dropping the two terms in the definition of $\psi$ which contain
integrals. (Recall that $\psi(b,k)\ge0$ for all $b$ and $k$, which
explains why we may consider only a subset of the $k$ and $b$, and
that the sum of the two integral terms sum is positive. Both facts
are explained in the discussion in the beginning of the proof of the
case $\tau>2$, around formula (\ref{eq:useful})). This gives 
\[
\psi(b,k)\ge(a-1+b)^{2}\binom{n}{b}e^{-t(ab+b)}y^{a+b+2}\apprge n\cdot\frac{2^{n}}{\sqrt{n}}\cdot e^{-n\tau/4}.
\]
Since there are $\approx\sqrt{n}$ terms with this estimate we get
\begin{equation}
\sum_{b,k}\psi(b,k)\apprge n2^{n}e^{-n\tau/4}.\label{eq:tau<2 lower}
\end{equation}
Notice that the exponential terms are exactly $\exp(-n\mu(\frac{1}{2}))$.

For the upper bound we need to split into two cases: $\beta=b/n<\frac{1}{2}-(2-\tau)/10$
and the complement. We start with the first, for which we can use
(\ref{eq:upper one b}) as is (its proof did not use any assumptions
on $\tau$, $b$ or $k$). We get
\[
\psi(b,k)\apprle n^{2}\binom{n}{b}e^{-tb(n-b)}\apprle n^{3/2}\exp(-n\mu(\beta))
\]
and because $\beta<\frac{1}{2}$ and $\mu$ is strictly decreasing,
$\psi(b,k)\apprle n2^{n}e^{-n\tau/4}(1-c)^{n}$ for some $c(\tau)>0$.
Thus these terms are negligible.

The other case is $\beta\ge\frac{1}{2}-(2-\tau)/10$. In this case
\[
e^{-\tau\kappa}>1-\frac{\kappa}{1-b/n}
\]
which means that for all $k$ sufficiently large, $y>x_{\max}+c(\tau)\kappa$
(recall that $x_{\max}=1-(k-1)/(n-b)$). This gives an estimate for
the integrals in the definition of $\psi$ by their value at $y$
times the length of the integration interval (up to a constant, for
$k\apprle1$). For example,
\begin{align*}
\int_{y}^{1}x^{a+1}(1-x)^{k-1} & \apprle(1-y)\cdot y^{a+1}(1-y)^{k-1}\apprle\frac{k}{n}x_{\max}^{a+1}(1-x_{\max})^{k-1}(1-c)^{k}\\
 & \stackrel{\textrm{\clap{(\ref{eq:integral max integrand})}}}{\apprle}\frac{k}{n}\cdot\frac{n}{\sqrt{k}}\cdot\frac{(a+1)!(k-1)!}{(a+k+1)!}(1-c)^{k}
\end{align*}
A similar estimate holds for the other integral. The remaining summand
in the definition of $\psi$ has a similar estimate, but without the
length of the integration interval, i.e.\ the term $\frac{k}{n}$
above. We get
\begin{align*}
\psi(b,k) & \apprle(a+1-b)\binom{n}{k-1,a+1,b}e^{-tb(a+1)}y^{b}\cdot\frac{n}{\sqrt{k}}\frac{(a+1)!(k-1)!}{(a+k+1)!}(1-c)^{k}\cdot\\
 & \qquad\qquad\left[\frac{(a+1-b)y}{k}+\frac{k}{n}(a+1-b+k)+\frac{k}{n}|a+1-b-k|\right]\\
 & \apprle(n+1-2b)^{2}\binom{n}{b}e^{-tb(n-b)}\frac{1}{\sqrt{k}}(1-c)^{k}\left(\frac{1}{k}+\frac{k}{n}\right)\\
 & \stackrel{\textrm{\clap{(\ref{eq:enter the mu})}}}{\apprle}\frac{1}{\sqrt{n}}(n+1-2b)^{2}\exp(-n\mu(\beta))(1-c)^{k}.
\end{align*}
Finally, writing $\mu(\beta)\ge\mu(\frac{1}{2})-c(1-2\beta)^{2}$
(compare to (\ref{eq:mu quadratic})) we get
\begin{align}
\sum_{b}\psi(b,k) & \apprle\frac{1}{\sqrt{n}}2^{n}e^{-\tau n/4}(1-c)^{k}\sum_{b=0}^{n/2}(n+1-2b)^{2}\exp\big(-c(1-2b/n)^{2}\big)\nonumber \\
 & \apprle n2^{n}e^{-n\tau/4}(1-c)^{k}.\label{eq:tau<2 upper}
\end{align}
where the first sum is only over $b$ such that $b/n\ge\frac{1}{2}-(2-\tau)/10$,
and of course such that $b\le\frac{1}{2}(n-k)$. Again, we omit the
details of the last inequality. Since the sum over the other values
of $b$ is negligible, we see that (\ref{eq:tau<2 upper}) in fact
holds (perhaps with a different value of $c$) even when the first
sum is over all $b$. The proof is now finished: we write 
\[
m^{2}=n\sum_{k=1}^{n}k\frac{\sum_{b}\psi(b,k)}{\sum_{b,l}\psi(b,l)}\stackrel{\textrm{(\ref{eq:tau<2 lower},\ref{eq:tau<2 upper})}}{\apprle}n\sum_{k=1}^{n}k(1-c)^{k}\apprle n
\]
as needed.
\end{proof}

\appendix

\section{The spontaneous and the residual magnetisations}

The purpose of this appendix is to show that for $t$ such that the
expected square magnetisation $m$ satisfies $m(t)\apprge n$ we also
have that the residual magnetisation $m^{*}$ (recall its definition
(\ref{eq:mstar})) satisfies $m^{*}(t)>0$. Recall that $Z=\mathbb{E}(2^{\alpha})$
and $m^{2}=(1/Z)\mathbb{E}((\sum k^{2}\alpha_{k})2^{\alpha})$. Let
$\varepsilon$ be some positive number and write 
\begin{align*}
\sum_{k}k^{2}\alpha_{k} & =\sum_{k\ge\varepsilon n}k^{2}\alpha_{k}+\sum_{k<\varepsilon n}k^{2}\alpha_{k}\le n\Big(\sum_{k\ge\varepsilon n}k\alpha_{k}\Big)+\varepsilon n\sum_{k<\varepsilon n}k\alpha_{k}\\
 & =n\Big(\sum_{k\ge\varepsilon n}k\alpha_{k}\Big)\mathbbm{1}\bigg\{\sum_{k\ge\varepsilon n}\alpha_{k}>0\bigg\}+\varepsilon n\sum_{k<\varepsilon n}k\alpha_{k}\\
 & \le n^{2}\mathbbm{1}\bigg\{\sum_{k\ge\varepsilon n}\alpha_{k}>0\bigg\}+\varepsilon n^{2}
\end{align*}
where in the second inequality we used $\sum k\alpha_{k}=n$. Now
multiply by $2^{\alpha}$ and take expectations and get
\[
\mathbb{E}\Big(\mathbbm{1}\bigg\{\sum_{k\ge\varepsilon n}\alpha_{k}>0\bigg\}2^{\alpha}\Big)\ge\frac{1}{n^{2}}Zm^{2}-\varepsilon Z
\]
Recall our assumption that $m>cn$ for some $c>0$. Choose $\varepsilon=\frac{1}{2}c^{2}$
and get
\[
\mathbb{E}\Big(\mathbbm{1}\bigg\{\sum_{k\ge\varepsilon n}\alpha_{k}>0\bigg\}2^{\alpha}\Big)\ge\frac{1}{2}c^{2}Z.
\]
Now examine the definition of $m^{*}$. Fix some $M$ and let $n>M/\varepsilon$.
Then 
\[
\sum_{k=M+1}^{n}\frac{k}{n}\alpha_{k}\ge\sum_{k\ge\varepsilon n}\frac{k}{n}\alpha_{k}\ge\varepsilon\mathbbm{1}\bigg\{\sum_{k\ge\varepsilon n}\alpha_{k}>0\bigg\}.
\]
Again multiplying by $2^{\alpha}$ and taking expectations gives
\[
\sum_{k=M+1}^{n}\frac{k}{n}\mathbb{E}(\alpha_{k}2^{\alpha(\pi(\beta))})\ge\varepsilon\mathbb{E}\bigg(\mathbbm{1}\bigg\{\sum_{k\ge\varepsilon n}\alpha_{k}>0\bigg\}2^{\alpha}\bigg)\ge\varepsilon\cdot\frac{1}{2}c^{2}Z=\varepsilon^{2}Z
\]
or

\[
\varliminf_{n\to\infty}\frac{1}{Z_{n}(\beta)}\sum_{k=M+1}^{n}\frac{k}{n}\mathbb{E}(\alpha_{k}2^{\alpha(\pi(\beta))})\ge\varepsilon^{2}.
\]
Since $\varepsilon$ did not depend on $M$, we get $m^{*}\ge\varepsilon^{2}$,
as needed.

\end{document}